\documentclass[a4paper,11pt]{article}

\usepackage{makeidx}
\usepackage{amscd}
\usepackage{amsmath}
\usepackage{amssymb}
\usepackage{amsthm}

\theoremstyle{plain}
\newtheorem{theorem}{Theorem}[section]

\newtheorem{lem}[theorem]{Lemma}
\newtheorem{cla}[theorem]{Claim}
\newtheorem{prop}[theorem]{Proposition}
\newtheorem{conj}[theorem]{Conjecture}

\theoremstyle{definition}
\newtheorem{defn}[theorem]{Definition}
\newtheorem{rem}[theorem]{Remark}
\newtheorem{exa}[theorem]{Example}

\newcommand{\PP}{\mathbb P}
\newcommand{\Z}{\mathbb Z}
\newcommand{\Q}{\mathbb Q}
\newcommand{\R}{\mathbb R}
\newcommand{\C}{\mathbb C}

\newcommand{\Ob}{\mathfrak{Ob}}
\newcommand{\id}{\ensuremath{\mathop{\mathrm{id}}}}
\def\t{\,{}^t\!}
\newcommand{\rk}{\ensuremath{\operatorname{rank}}}

\newcommand{\Pic}{\mathop{\mathrm{Pic}}\nolimits}
\newcommand{\Ext}{\mathop{\mathrm{Ext}}\nolimits}
\newcommand{\Hom}{\mathop{\mathrm{Hom}}\nolimits}

\newcommand{\RHom}{\mathop{\mb R\mathrm{Hom}}\nolimits}
\newcommand{\mcRHom}{\mathop{\mb R\mathcal{H}om}\nolimits}
\newcommand{\RGamma}{\mathop{\mb R\Gamma}\nolimits}
\newcommand{\Lotimes}{\stackrel{\mb L}{\otimes}}
\newcommand{\Lboxtimes}{\stackrel{\mb L}{\boxtimes}}

\newcommand{\End}{\mathop{\mathrm{End}}\nolimits}
\newcommand{\Spec}{\operatorname{Spec}}
\newcommand{\GL}{GL}
\newcommand{\Coh}{\operatorname{Coh}}
\newcommand{\QCoh}{\operatorname{QCoh}}
\newcommand{\Auteq}{\operatorname{Auteq}}
\newcommand{\module}{\operatorname{mod}}
\newcommand{\Aut}{\operatorname{Aut}} 
\newcommand{\Sym}{Sym}
\newcommand{\nef}{nef}
 
\newcommand{\mc}{\mathcal}
\newcommand{\mb}{\mathbb}
\newcommand{\mk}{\mathfrak}
\newcommand{\NI}{\noindent}
\newcommand{\Supp}{\ensuremath{\operatorname{Supp}}}
\renewcommand{\labelenumi}{(\roman{enumi})}
\renewcommand{\labelenumii}{(\arabic{enumii})}

\newcommand{\Span}[1]{\left<#1\right>}
\newcommand{\Conv}[1]{Conv\left<#1\right>}

\def\mbi#1{\boldsymbol{#1}}
\newcommand{\Div}{\operatorname{div}}
\newcommand{\lce}{\lceil}
\newcommand{\rce}{\rceil}
\newcommand{\lfl}{\lfloor}
\newcommand{\rfl}{\rfloor}
\title{Exceptional collections on toric Fano threefolds and birational geometry
}
\author{Hokuto Uehara
}

\date{}

\begin{document}

\maketitle

\begin{abstract}
Bondal's conjecture 
states that the Frobenius push-forward of the structure sheaf 
$\mc O_X$ generates the derived category $D^b(X)$ for 
smooth projective toric varieties $X$. 

Bernardi and Tirabassi exhibit a full strong exceptional collection 
consisting of line bundles on  
smooth toric Fano $3$-folds
assuming Bondal's conjecture.
In this article, we prove Bondal's conjecture for 
smooth toric Fano $3$-folds
and improve upon their result
using birational geometry. 
\end{abstract}

\footnote{Keywords; Derived categories, toric Fano threefolds, birational geometry, exceptional collection}


\section{Introduction}
A full strong exceptional collection of a triangulated category can be thought of 
as the categorical analogue of a finite orthonormal basis of a vector space.
For the derived category $D^b(X)$ of coherent sheaves on a smooth projective variety $X$,
such a collection rarely exists, 
but if it exists, the derived category $D^b(X)$ is equivalent to the derived category 
$D^b(\module A)$ of the category $\module A$ of finitely generated modules over 
a finite dimensional algebra $A$.  
 
For any smooth toric DM orbifold $X$, 
Kawamata shows that there is a full, but not 
necessarily strong, exceptional collection on $X$ \cite{Ka06}.
Furthermore full strong exceptional collections on toric varieties
are studied by many people (cf.~\cite{BT10,CM10,CM,CRM11,DLM10,HP08,LM10,IU09}).

We can define an endomorphism $F_m$ ($m\in \Z_{>0}$)
called the \emph{Frobenius map} 
on any toric variety over a field of any characteristic
(some people also call it a \emph{multiplication map}).
It is also known that for smooth complete toric varieties $X$, 
$F_{m*}\mc O_X$ splits into line bundles and Thomsen \cite{Th00} 
finds an algorithm to compute the set of all direct summands of it. 
We denote the set by $\mk D_X$ for a sufficiently divisible integer $m$.
On the other hand, Bondal's conjecture predicts that
the set $\mk D_X$ classically generates the derived category $D^b(X)$. 
So sometimes, for instance in the case $|\mk D_X|=\rk K(X)$, 
it becomes a candidate of a full strong exceptional collection 
 consisting of line bundles on smooth projective toric varieties $X$. 

Bernardi and Tirabassi exhibit such 
collections on all eighteen smooth toric Fano $3$-folds by using 
Frobenius maps \cite{BT10}.
Using birational geometry, we obtain a stronger result. Precisely we show the following.

\begin{theorem}\label{thm:main1}
For sixteen smooth toric Fano $3$-folds $X$ over $\C$, 
the set $\mk D_X$ becomes a full strong exceptional collection
(after choosing an appropriate order). 
For the remaining two cases, $(4)$ and $(11)$ in Theorem $\ref{thm:18fano}$, 
we present a proper subset of $\mk D_X$ which 
becomes a full strong exceptional collection.
\end{theorem}

Note that this theorem implies that Bondal's conjecture is true for smooth toric Fano $3$-folds. 
The strategy to prove Theorem \ref{thm:main1} is as follows;

\paragraph{Step 1.}
Let $f\colon X\to Y$ be an extremal birational contractions between 
smooth toric Fano $3$-folds. Assume that $\mk D_X$ 
forms a full strong exceptional collection. 
Then so is $\mk D_Y$. 
This is done by Lemmas \ref{lem:f_*} and \ref{lem:fano3}.

\paragraph{Step 2.}
By Step 1, it is enough to show that $\mk D_X$ forms a full 
strong exceptional collection only
for (birationally) maximal Fano $3$-folds $X$, 
namely in (11), (17) and (18) 
in Theorem \ref{thm:18fano}. Unfortunately, 
in the case (11), $\mk D_X$ does not form a strong 
exceptional collection. Instead
we can find a subset $\mk D_{\nef}$ of 
$\mk D_X$ which becomes a full strong exceptional collection.
Then, as in Step 1, an inductive argument works 
in the case $X$ in (11).

\paragraph{Step 3.} 
To check the strongness of the chosen set in Step 2,
it is enough to check the dual of line bundles 
in the set are nef (Lemma \ref{lem:vani}). 
This is easily done by observing Figure \ref{Fano11&18}.
To check the fullness in Step 2, 
we prove Bondal's conjecture
in our situation by rather tedious, but elementary calculation.
This step is done in \S \ref{section:Examples} and \S \ref{sec:Bondal}.

\paragraph{} 
In \cite{BT10},
 Bernardi and Tirabassi check a similar statement to Theorem \ref{thm:main1} 
only in the cases
(9), (11), (14), (15) and (16) separately, since the existence of 
a full strong exceptional collection, not necessarily coming from 
$\mk D_X$, was already known in the remaining cases. 
In their proof of the strongness, a rather long and explicit calculation is presented,
using the result in \cite{BH09}.   
Moreover they deduce fullness for their collections only after assuming the conjecture of  Bondal (see Remark \ref{rem:Bondal}).

Dubrobin conjectures that
for a smooth projective variety $X$,
the quantum cohomology of $X$ is semi-simple if and only if 
$X$ is a smooth Fano variety with a full exceptional collection.
Although his conjecture turns out to be wrong, it is still believed that there is a relationship 
between the existence of full exceptional collections on $X$ and its quantum cohomology 
(cf.~\cite{Ba04}.) 
Furthermore several people conjectured
that every smooth toric Fano variety has a
 full strong exceptional collection
consisting of line bundles \cite{BH09,CM10}.
But recently Efimov provides a counterexample to this conjecture 
\cite{Ef10}.  At least in the $3$-dimensional case, the conjecture is true
by Theorem \ref{thm:main1}.  

It should be pointed out that there is a smooth projective toric (not Fano) surface 
which does not possess any
full strong exceptional collections consisting of line bundles (\cite{HP06,HP08}),
and it is also worthwhile to mention that there is a smooth 
toric Fano variety $X$ such that we cannot 
choose full strong exceptional collections from the set $\mk D_X$ 
\cite{LM10}.

Because our exceptional collections consist of line bundles, 
the corresponding quivers, Gram matrices etc., 
should be rather easy to  be computed. 
Furthermore our collection exhibits nice properties, as it behaves well as in Step 1 above. 

The structure of this paper is as follows:
In \S \ref{sec:generator}, we give some basic definitions on the derived 
categories of coherent sheaves. 
In \S \ref{sec:toric},
we explain several notion on toric varieties and cite some useful
results for smooth toric Fano $3$-folds. We also explain how 
to determine the set $\mk D_X$, following Thomsen. 
In \S \ref{section:Examples}, we actually determine it for several
toric varieties. 
In \S \ref{sec:Bondal}, we prove Bondal's conjecture 
for maximal smooth toric Fano $3$-folds. 
In \S \ref{sec:birational}, we accomplish Step 1 above and give the proof 
of Theorem \ref{thm:main1}. In Theorem \ref{thm:surface}  
we also obtain a similar result 
in the surface case to Theorem \ref{thm:main1}.

\paragraph{Notation and conventions}
For a smooth variety $X$, we denote the bounded 
 derived category of coherent sheaves on $X$
by $D^b(X)$. \emph{$T$-invariant} is an abbreviation of 
\emph{torus invariant}. For objects $\mc E, \mc F\in D^b(X)$, 
we define 
$$
\Hom _X^i(\mc E,\mc F):=\Hom _{D^b(X)}(\mc E,\mc F[i]).
$$
We work over $\C$ for simplicity.
 
We denote by $M(l,m)$ the space of $l\times m$ matrices defined over $\Z$, and 
${}^tA$ is the transpose of a matrix $A\in M(l,m)$

 
\section{Generators of derived categories}\label{sec:generator}
In this section, 
we give several basic definitions on triangulated 
categories and derived categories of coherent sheaves. 

\begin{defn}
Let $\mc S=\{\mc S_i\}$ be a set of objects in a triangulated category $\mc D$. 
\begin{enumerate}
\item
We denote by  $\Span{\mc S}$   
the smallest triangulated subcategory of $\mc D$ containing all $\mc S_i$, 
closed under isomorphisms and direct summands.
For a triangulated subcategory $\mc C$ of $\mc D$, 
we denote by $\mc C^{\perp}$ the full triangulated 
subcategory of $\mc D$ whose objects $\mc F$ 
satisfy the property $\Hom _{\mc D}(C,\mc F)=0$ for all $C\in \mc C$. 
\item
We say that $\mc S$ \emph{classically generates} $\mc D$ 
if $\Span{\mc S}=\mc D$. 
We also call $\mc S$ a \emph{classical generator} of $\mc D$.
\item
We say that $\mc S$ \emph{generates} $\mc D$ if $\Span{\mc S}^\perp=0$.
We also call $\mc S$ a \emph{generator} of $\mc D$.
\end{enumerate}
\end{defn}

Let $X$ be a smooth complete variety over $\C$.

\begin{defn} 
\begin{enumerate}
\item
An object $\mc E\in D^b(X)$ is called \emph{exceptional} if it satisfies 
$$
\Hom _X^i(\mc E,\mc E)=
\begin{cases}
\C &\mbox{$i=0$}\\
0 &\mbox{otherwise}.
\end{cases}
$$ 
\item
An ordered set $(\mc  E_1,\ldots,\mc E_n)$ 
of exceptional objects is called an \emph{exceptional collection} 
if the following condition holds;
$$\Hom ^i_X(\mc E_k,\mc E_j)=0$$ 
for all $k>j$ and all $i$. When we say that a finite set $\mc S$
of objects is an exceptional collection, it means that $\mc S$ 
forms an exceptional collection after choosing an appropriate order. 
\item 
An exceptional collection $(\mc  E_1,\ldots,\mc E_n)$ is called 
\emph{strong} if 
$$\Hom ^i_X(\mc E_k,\mc E_j)=0$$ 
for all $k,j$ and $i\ne 0$.
\item
An exceptional collection $(\mc E_1,\ldots,\mc E_n)$ is called 
\emph{full} if 
$$\Span{\mc  E_1,\ldots,\mc E_n}=D^b(X).$$ 
\end{enumerate}
\end{defn}

\begin{rem}
If $X$ has a full exceptional collection consisting of
$n$ exceptional objects, the rank of its $K$-group $K(X)$ is $n$ (\cite{BP94}).
Furthermore it is known that the rank of $K$-group is the 
number of the maximal cones in the fan
for smooth projective toric varieties (cf.~\cite[Theorem in \S 5.2.]{Fu93}). 
For $3$-dimensional smooth projective 
toric varieties $X$, we can see $\rk K(X)=2\rho (X)+2$.
\end{rem}


\section{Toric varieties}\label{sec:toric}
Throughout this section, we use the following notation.
Let $N=\Z ^n$ be a lattice of rank $n$ and $M$ its dual.
A fan $\Delta$ in $N_{\R}=N\otimes _\Z\R$ consists of a finite number of rational  strongly convex polyhedral cones 
in $N_{\R}$, and it 
determines a toric variety $X=X(\Delta)$.
We define 
$
\mc V(\Delta)
$ 
to be the set of primitive generators of $1$-dimensional cones in $\Delta$.

For a cone $\sigma$ in $\Delta$,
put 
$R_\sigma =\C\bigl[ \chi ^{\mbi u} \bigm | 
\mbi u \in M\cap \sigma ^\vee\bigr]$, 
where 
$\bigl\{ \chi ^{\mbi u} \bigm | 
\mbi u \in M\cap \sigma ^\vee \bigr \}$ is a basis of $\C$-vector space $R_{\sigma}$. 
We can define
an obvious multiplication on $R$ as usual (cf.~\cite[page 15]{Fu93}).
Then the affine toric variety $U_{\sigma}$ corresponds to $\sigma$ is just 
$\Spec R_{\sigma}$, and
the rational function field of the toric varieties $X$
is just 
$\C\bigl[ \chi ^{\mbi u} \bigm | \mbi u \in M\bigr].$ 

\subsection{Double $\Z$-weight}
According to \cite{Od88}, we introduce the notion of 
doubly $\Z$-weighted 
triangulations of a $2$-sphere. For simplicity, 
we restrict to the case $n=3$, namely $N=\Z^3$ below.
We can obviously identify the set of 
half lines starting from the origin $\mbi 0$ of $N_\R$ with 
$$
S^{2}:=(N_\R \backslash \{\mbi 0\})/\R_{>0} .
$$ 
Let 
$$
\pi\colon N_\R \backslash \{\mbi 0\}\to S^{2}
$$
be the projection.
We call $\pi (v)$ a \emph{rational point} of $S^{2}$ corresponding 
to a primitive element $v\in N$, and $v$ is called the \emph{$N$-weight} of 
the rational point $\pi(v)$.
For a cone 
$\sigma=\R_{\ge 0}\mbi v_{1}+\cdots +\R_{\ge 0}\mbi v_{s}\in \Delta$
($\mbi v_i\in \mc V(\Delta)$), 
$\pi (\sigma \backslash \{\mbi 0\})$ 
 is a convex spherical cell in $S^{2}$ with rational points
 $\pi(\mbi v_{1}),\ldots,\pi(\mbi v_{s})$ as vertices.
Thus for a fan $\Delta$, we get a convex spherical cell decomposition 
$$
\bigl\{\pi (\sigma\backslash \{\mbi 0\} \bigm| \sigma\in \Delta \bigr\}
$$ of 
$\pi (|\Delta| \backslash \{\mbi 0\})$. 

Suppose that a fan $\Delta$ is complete and non-singular,
which is equivalent to the condition that the corresponding toric variety 
$X=X(\Delta)$ is proper and smooth. Then we get a simplicial 
cell decomposition
of $S^{2}$. Moreover, for each $3$-dimensional cone $\sigma\ \in \Delta$,
 the corresponding  spherical $2$-simplex 
$\pi (\sigma\backslash \{\mbi 0\})$ has vertices whose $N$-weights 
$\mbi v_{1},\mbi v_2,\mbi v_{3}$ form a $\Z$-basis of $N$.
For each $2$-dimensional cone $\tau \in \Delta$, there are exactly two 
$3$-dimensional cone $\sigma,\sigma '\in \Delta$ such that 
$\sigma \cap \sigma'=\tau$. In this case, the sets 
$\{\mbi v,\mbi v_2,\mbi v_3\}$ and   
$\{\mbi v',\mbi v_2,\mbi v_3\}$ of $N$-weights for the vertices of 
$\pi(\sigma \backslash \{\mbi 0\})$ and 
$\pi(\sigma' \backslash \{\mbi 0\})$, respectively,
are $\Z$-bases of $N$. Moreover we have
$$
\mbi v+\mbi v'+\alpha_2\mbi v_2+\alpha_3\mbi v_3=\mbi 0
$$
for $\alpha_j\in \Z$ uniquely determined by $\tau$. 
For 
$$\rho=\R_{\ge 0}\mbi v, \rho'=\R_{\ge 0}\mbi v',
\rho_j=\R_{\ge 0}\mbi v_j\in \Delta  \quad \mbox{ ($j=2,3$) },$$
let $D,D',D_j$ are the corresponding $T$-invariant divisors. 
Then it is known (cf.~\cite[page 81]{Od88}) that we have
\begin{equation}\label{eqn:inter}
\alpha_j=( D_j\cdot D_2\cdot D_3).
\end{equation}
We then endow the edge $\pi(\mbi v_2),\pi(\mbi v_3)$ with the \emph{double
$\Z$-weight} $\alpha_2,\alpha_3$, 
where we place $\alpha_2$ (resp. $\alpha_3$)
on the side of the vertex $\pi(\mbi v_2)$ (resp. $\pi(\mbi v_3)$) 
as in Figure \ref{double}.  For simplicity, here and henceforth 
we always denote a rational point $\pi(\mbi v)$ 
by its $N$-weight $\mbi v$ in figures of 
double $\Z$-weights. 
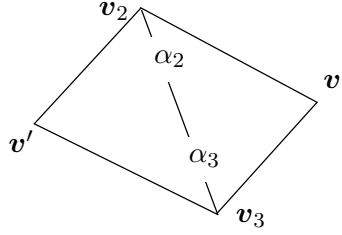
\begin{figure}[h]
\[
\unitlength 0.1in
\begin{picture}( 19.3700, 11.4800)( -0.5500,-17.4300)
%
\special{pn 8}%
\special{pa 974 678}%
\special{pa 1368 1744}%
\special{fp}%
\special{pa 1368 1744}%
\special{pa 418 1270}%
\special{fp}%
\special{pa 418 1270}%
\special{pa 968 668}%
\special{fp}%
\special{pa 968 664}%
\special{pa 1882 1158}%
\special{fp}%
\special{pa 1882 1158}%
\special{pa 1362 1740}%
\special{fp}%
\put(15.4000,-17.6000){\makebox(0,0){$\mbi v_3$}}%
\put(8.3000,-6.8000){\makebox(0,0){$\mbi v_2$}}%
%
\special{pn 8}%
\special{sh 0}%
\special{pa 930 820}%
\special{pa 1206 820}%
\special{pa 1206 1050}%
\special{pa 930 1050}%
\special{pa 930 820}%
\special{ip}%
\put(11.2000,-9.3000){\makebox(0,0){$\alpha_2$}}%
%
\special{pn 8}%
\special{sh 0}%
\special{pa 1110 1340}%
\special{pa 1386 1340}%
\special{pa 1386 1570}%
\special{pa 1110 1570}%
\special{pa 1110 1340}%
\special{ip}%
\put(13.0000,-14.5000){\makebox(0,0){$\alpha_3$}}%
\put(19.6000,-10.5000){\makebox(0,0){$\mbi v$}}%
\put(3.5000,-13.6000){\makebox(0,0){$\mbi v'$}}%
\end{picture}%
\]
\caption{Double $\Z$-weight}\label{double} 
\end{figure} 

We have normal bundle sequences;
\begin{align*}
0\to &\mc N_{C/D_2}\to \mc N_{C/X} \to \mc N_{D_2/X}|_C\to 0\\
0\to &\mc N_{C/D_3}\to \mc N_{C/X} \to \mc N_{D_3/X}|_C\to 0,
\end{align*}
where $C\cong \mb P^1$ is the $T$-invariant curve 
corresponding to the cone $\tau$.
Then we know as above $\mc N_{C/D_j}\cong \mc O_{\PP^1}(-\alpha_ j)$
and $\mc N_{D_j/X}|_C \cong \mc O_{\PP^1}(-\alpha_ {j'})$, where
$\{j,j'\}=\{2,3\}$. Combining both sequences,
we conclude that they split and so we have
\begin{equation}\label{eqn:normal}
\mc N_{C/X}\cong 
\mc O_{\PP^1}(-\alpha _2)\oplus \mc O_{\PP^1}(-\alpha _3).
\end{equation}

We show in Figure \ref{blowup} 
the change of double $\Z$-weights under the blowing-up
along a $T$-invariant curve  \cite[page 90]{Od88}. 
The segment attached to a oval
corresponds to the curve, and the vertex with a dark gray small circle
corresponds to the exceptional divisor.  
Figure \ref{blowup} will be used 
to find the centers and the exceptional divisors of blowing-ups
in  Figure \ref{pic18}.
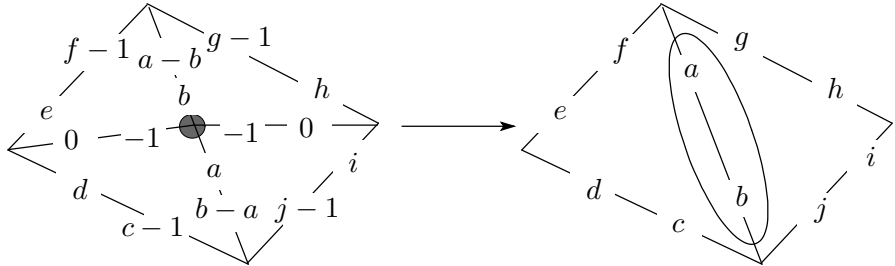
\begin{figure}[h]
\[
\unitlength 0.1in
\begin{picture}( 45.7800, 13.6000)(  4.6000,-20.2000)
%
\special{pn 8}%
\special{pa 3850 680}%
\special{pa 4364 2020}%
\special{fp}%
\special{pa 4364 2020}%
\special{pa 3120 1424}%
\special{fp}%
\special{pa 3120 1424}%
\special{pa 3840 666}%
\special{fp}%
\special{pa 3840 660}%
\special{pa 5038 1284}%
\special{fp}%
\special{pa 5038 1284}%
\special{pa 4358 2016}%
\special{fp}%
%
\special{pn 8}%
\special{sh 0}%
\special{pa 3876 928}%
\special{pa 4108 928}%
\special{pa 4108 1116}%
\special{pa 3876 1116}%
\special{pa 3876 928}%
\special{ip}%
\put(40.0000,-10.1200){\makebox(0,0){$a$}}%
%
\special{pn 4}%
\special{sh 0.600}%
\special{ar 1414 1300 66 62  0.0000000 6.2831853}%
%
\special{pn 8}%
\special{sh 0}%
\special{pa 4136 1582}%
\special{pa 4368 1582}%
\special{pa 4368 1770}%
\special{pa 4136 1770}%
\special{pa 4136 1582}%
\special{ip}%
\put(42.5900,-16.6500){\makebox(0,0){$b$}}%
%
\special{pn 8}%
\special{sh 0}%
\special{pa 3822 1740}%
\special{pa 4054 1740}%
\special{pa 4054 1928}%
\special{pa 3822 1928}%
\special{pa 3822 1740}%
\special{ip}%
\put(39.3000,-18.1800){\makebox(0,0){$c$}}%
%
\special{pn 8}%
\special{sh 0}%
\special{pa 3368 1532}%
\special{pa 3600 1532}%
\special{pa 3600 1720}%
\special{pa 3368 1720}%
\special{pa 3368 1532}%
\special{ip}%
\put(34.9300,-16.2600){\makebox(0,0){$d$}}%
%
\special{pn 8}%
\special{sh 0}%
\special{pa 3196 1146}%
\special{pa 3428 1146}%
\special{pa 3428 1334}%
\special{pa 3196 1334}%
\special{pa 3196 1146}%
\special{ip}%
\put(33.2000,-12.3000){\makebox(0,0){$e$}}%
%
\special{pn 8}%
\special{sh 0}%
\special{pa 3510 818}%
\special{pa 3740 818}%
\special{pa 3740 1008}%
\special{pa 3510 1008}%
\special{pa 3510 818}%
\special{ip}%
\put(36.3300,-9.0300){\makebox(0,0){$f$}}%
%
\special{pn 8}%
\special{sh 0}%
\special{pa 4136 770}%
\special{pa 4368 770}%
\special{pa 4368 958}%
\special{pa 4136 958}%
\special{pa 4136 770}%
\special{ip}%
\put(42.5900,-8.5400){\makebox(0,0){$g$}}%
%
\special{pn 8}%
\special{sh 0}%
\special{pa 4622 1016}%
\special{pa 4854 1016}%
\special{pa 4854 1206}%
\special{pa 4622 1206}%
\special{pa 4622 1016}%
\special{ip}%
\put(47.4500,-11.0100){\makebox(0,0){$h$}}%
%
\special{pn 8}%
\special{sh 0}%
\special{pa 4806 1374}%
\special{pa 5036 1374}%
\special{pa 5036 1562}%
\special{pa 4806 1562}%
\special{pa 4806 1374}%
\special{ip}%
\put(49.2900,-14.5700){\makebox(0,0){$i$}}%
%
\special{pn 8}%
\special{sh 0}%
\special{pa 4546 1650}%
\special{pa 4778 1650}%
\special{pa 4778 1838}%
\special{pa 4546 1838}%
\special{pa 4546 1650}%
\special{ip}%
\put(46.7000,-17.3500){\makebox(0,0){$j$}}%
%
\special{pn 8}%
\special{pa 1190 680}%
\special{pa 1704 2020}%
\special{fp}%
\special{pa 1704 2020}%
\special{pa 460 1424}%
\special{fp}%
\special{pa 460 1424}%
\special{pa 1180 666}%
\special{fp}%
\special{pa 1180 660}%
\special{pa 2378 1284}%
\special{fp}%
\special{pa 2378 1284}%
\special{pa 1698 2016}%
\special{fp}%
%
\special{pn 8}%
\special{sh 0}%
\special{pa 1162 1740}%
\special{pa 1394 1740}%
\special{pa 1394 1928}%
\special{pa 1162 1928}%
\special{pa 1162 1740}%
\special{ip}%
\put(12.1000,-18.3000){\makebox(0,0){$c-1$}}%
%
\special{pn 8}%
\special{sh 0}%
\special{pa 708 1532}%
\special{pa 940 1532}%
\special{pa 940 1720}%
\special{pa 708 1720}%
\special{pa 708 1532}%
\special{ip}%
\put(8.3300,-16.2600){\makebox(0,0){$d$}}%
%
\special{pn 8}%
\special{sh 0}%
\special{pa 536 1146}%
\special{pa 768 1146}%
\special{pa 768 1334}%
\special{pa 536 1334}%
\special{pa 536 1146}%
\special{ip}%
\put(6.6000,-12.3000){\makebox(0,0){$e$}}%
%
\special{pn 8}%
\special{sh 0}%
\special{pa 850 818}%
\special{pa 1080 818}%
\special{pa 1080 1008}%
\special{pa 850 1008}%
\special{pa 850 818}%
\special{ip}%
\put(9.2000,-9.0300){\makebox(0,0){$f-1$}}%
%
\special{pn 8}%
\special{sh 0}%
\special{pa 1476 770}%
\special{pa 1708 770}%
\special{pa 1708 958}%
\special{pa 1476 958}%
\special{pa 1476 770}%
\special{ip}%
%
\special{pn 8}%
\special{sh 0}%
\special{pa 1962 1016}%
\special{pa 2194 1016}%
\special{pa 2194 1206}%
\special{pa 1962 1206}%
\special{pa 1962 1016}%
\special{ip}%
\put(20.8500,-11.0100){\makebox(0,0){$h$}}%
%
\special{pn 8}%
\special{sh 0}%
\special{pa 2146 1374}%
\special{pa 2376 1374}%
\special{pa 2376 1562}%
\special{pa 2146 1562}%
\special{pa 2146 1374}%
\special{ip}%
\put(22.5000,-14.8800){\makebox(0,0){$i$}}%
%
\special{pn 8}%
\special{sh 0}%
\special{pa 1886 1650}%
\special{pa 2118 1650}%
\special{pa 2118 1838}%
\special{pa 1886 1838}%
\special{pa 1886 1650}%
\special{ip}%
\put(20.1000,-17.3500){\makebox(0,0){$j-1$}}%
%
\special{pn 8}%
\special{pa 460 1422}%
\special{pa 1422 1294}%
\special{fp}%
\special{pa 2372 1294}%
\special{pa 1422 1294}%
\special{fp}%
\put(16.6000,-8.4900){\makebox(0,0){$g-1$}}%
%
\special{pn 8}%
\special{sh 0}%
\special{pa 1046 1246}%
\special{pa 1276 1246}%
\special{pa 1276 1434}%
\special{pa 1046 1434}%
\special{pa 1046 1246}%
\special{ip}%
\put(11.5300,-13.6400){\makebox(0,0){$-1$}}%
%
\special{pn 8}%
\special{sh 0}%
\special{pa 1558 1218}%
\special{pa 1790 1218}%
\special{pa 1790 1406}%
\special{pa 1558 1406}%
\special{pa 1558 1218}%
\special{ip}%
\put(16.6600,-13.3700){\makebox(0,0){$-1$}}%
%
\special{pn 8}%
\special{sh 0}%
\special{pa 692 1290}%
\special{pa 924 1290}%
\special{pa 924 1478}%
\special{pa 692 1478}%
\special{pa 692 1290}%
\special{ip}%
\put(7.8900,-13.6900){\makebox(0,0){$0$}}%
%
\special{pn 8}%
\special{sh 0}%
\special{pa 1906 1228}%
\special{pa 2138 1228}%
\special{pa 2138 1416}%
\special{pa 1906 1416}%
\special{pa 1906 1228}%
\special{ip}%
\put(20.0400,-13.0600){\makebox(0,0){$0$}}%
%
\special{pn 8}%
\special{sh 0}%
\special{pa 1432 1444}%
\special{pa 1664 1444}%
\special{pa 1664 1632}%
\special{pa 1432 1632}%
\special{pa 1432 1444}%
\special{ip}%
\put(15.2900,-15.3300){\makebox(0,0){$a$}}%
%
\special{pn 8}%
\special{sh 0}%
\special{pa 1486 1660}%
\special{pa 1718 1660}%
\special{pa 1718 1848}%
\special{pa 1486 1848}%
\special{pa 1486 1660}%
\special{ip}%
\put(15.9400,-17.3100){\makebox(0,0){$b-a$}}%
%
\special{pn 8}%
\special{sh 0}%
\special{pa 1180 840}%
\special{pa 1412 840}%
\special{pa 1412 1028}%
\special{pa 1180 1028}%
\special{pa 1180 840}%
\special{ip}%
\put(12.9700,-9.4800){\makebox(0,0){$a-b$}}%
%
\special{pn 8}%
\special{sh 0}%
\special{pa 1290 1048}%
\special{pa 1522 1048}%
\special{pa 1522 1236}%
\special{pa 1290 1236}%
\special{pa 1290 1048}%
\special{ip}%
\put(13.7100,-11.3700){\makebox(0,0){$b$}}%
%
\special{pn 8}%
\special{pa 2500 1300}%
\special{pa 3080 1300}%
\special{fp}%
\special{sh 1}%
\special{pa 3080 1300}%
\special{pa 3014 1280}%
\special{pa 3028 1300}%
\special{pa 3014 1320}%
\special{pa 3080 1300}%
\special{fp}%
%
\special{pn 8}%
\special{pa 4310 1318}%
\special{pa 4320 1348}%
\special{pa 4330 1378}%
\special{pa 4340 1408}%
\special{pa 4348 1440}%
\special{pa 4356 1470}%
\special{pa 4364 1502}%
\special{pa 4370 1534}%
\special{pa 4376 1564}%
\special{pa 4382 1596}%
\special{pa 4386 1628}%
\special{pa 4390 1660}%
\special{pa 4394 1692}%
\special{pa 4394 1724}%
\special{pa 4394 1756}%
\special{pa 4394 1788}%
\special{pa 4390 1820}%
\special{pa 4382 1850}%
\special{pa 4370 1880}%
\special{pa 4350 1904}%
\special{pa 4320 1918}%
\special{pa 4288 1914}%
\special{pa 4260 1900}%
\special{pa 4234 1882}%
\special{pa 4210 1860}%
\special{pa 4188 1836}%
\special{pa 4168 1812}%
\special{pa 4148 1786}%
\special{pa 4130 1760}%
\special{pa 4114 1732}%
\special{pa 4098 1704}%
\special{pa 4082 1676}%
\special{pa 4068 1648}%
\special{pa 4052 1620}%
\special{pa 4040 1592}%
\special{pa 4026 1562}%
\special{pa 4012 1532}%
\special{pa 4000 1504}%
\special{pa 3988 1474}%
\special{pa 3978 1444}%
\special{pa 3968 1414}%
\special{pa 3958 1382}%
\special{pa 3948 1352}%
\special{pa 3938 1322}%
\special{pa 3930 1290}%
\special{pa 3922 1260}%
\special{pa 3914 1228}%
\special{pa 3908 1198}%
\special{pa 3902 1166}%
\special{pa 3896 1134}%
\special{pa 3892 1102}%
\special{pa 3888 1070}%
\special{pa 3884 1040}%
\special{pa 3884 1008}%
\special{pa 3884 976}%
\special{pa 3884 944}%
\special{pa 3888 912}%
\special{pa 3896 880}%
\special{pa 3908 850}%
\special{pa 3930 826}%
\special{pa 3958 814}%
\special{pa 3990 818}%
\special{pa 4020 832}%
\special{pa 4044 852}%
\special{pa 4068 872}%
\special{pa 4090 896}%
\special{pa 4110 920}%
\special{pa 4130 946}%
\special{pa 4148 972}%
\special{pa 4166 1000}%
\special{pa 4182 1028}%
\special{pa 4196 1056}%
\special{pa 4212 1084}%
\special{pa 4226 1112}%
\special{pa 4240 1142}%
\special{pa 4252 1170}%
\special{pa 4266 1200}%
\special{pa 4278 1230}%
\special{pa 4290 1258}%
\special{pa 4300 1290}%
\special{pa 4310 1318}%
\special{sp}%
\end{picture}%
\]
\caption{Change of double $\Z$-weights under the blowing-up}\label{blowup} 
\end{figure} 

\subsection{Toric Fano $3$-folds}
Smooth toric Fano $3$-folds are classified as follows. 


\begin{theorem}[\cite{Ba82,WW82}]\label{thm:18fano}
Up to isomorphism,
there are $18$ distinct Fano $3$-folds. 
Among them, each of $(11),(12),(14),(15),(16)$ and $(18)$ below
is obtained from one of the others by a finite succession
of equivariant blowing-ups.
\begin{enumerate}
\item[$(1)$]
$\PP ^3.$
\item[$(2)$]
$\PP^2 \times \PP ^1.$
\item[$(3)$]
The $\PP ^1$-bundle $\PP(\mc O_Y\oplus \mc O_Y(1))$
over $Y=\PP ^2$.
\item[$(4)$]
The $\PP ^1$-bundle $\PP(\mc O_Y\oplus \mc O_Y(2))$
over $Y=\PP ^2$.
\item[$(5)$]
The $\PP ^2$-bundle $\PP(\mc O_Y\oplus \mc O_Y \oplus \mc O_Y(1))$
over $Y=\PP ^1$.
\item[$(6)$]
$\PP ^1\times \PP^1\times \PP^1$.
\item[$(7)$]
The $\PP ^1$-bundle $\PP(\mc O_Y\oplus \mc O_Y(f_1+f_2))$
over $Y=\PP ^1\times \PP^1$, where $f_1$ and $f_2$ are fibers
of the two projections from $Y$ to $\PP^1$.
\item[$(8)$]
$\PP (\mc O_Y\oplus \mc O_Y(f_1-f_2))$ in the notation of $(7)$.
\item[$(9)$]
$\PP ^1\times \Sigma_1$ for the Hirzeburch surface $\Sigma_1$.
\item[$(10)$]
The $\PP ^1$-bundle $\PP(\mc O_Y\oplus \mc O_Y(s+f))$
over $Y=\Sigma_1$, where $f$ is a fiber of the $\PP^1$-bundle on $\Sigma_1$
and $s$ is the minimal section with $s^2=-1$. 
\item[$(13)$]
$\PP^ 1\times Y_2$, where $Y_2$ is the toric del Pezzo surface obtained from 
$\PP^ 2$ by the equivariant blowing-up at two of the $T$-invariant points.
\item[$(17)$]
$\PP^ 1\times Y_3$, where $Y_3$ is the toric del Pezzo surface obtained from 
$\PP^ 2$ by the equivariant blowing-up at the three $T$-invariant points.
\end{enumerate}

Their birational relations are described in Figure \ref{pic18a}.
There are just three maximal Fano $3$-folds, $(11)$, $(17)$ and $(18)$,
with respect to birational relations.

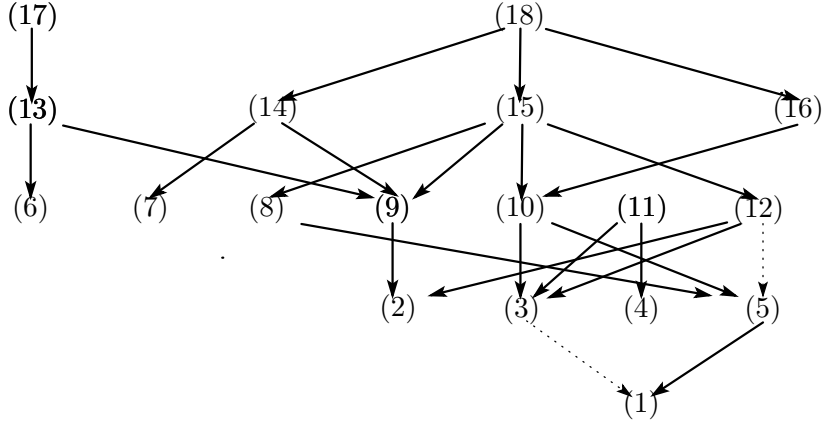
\begin{figure}[h]
\[
\unitlength 0.1in
\begin{picture}( 42.2900, 20.6500)(  5.9900,-27.0500)
\put(8.6900,-7.2500){\makebox(0,0){$(17)$}}%
\put(8.6900,-7.2500){\makebox(0,0){$(17)$}}%
\put(8.6900,-12.1600){\makebox(0,0){$(13)$}}%
%
\special{pn 13}%
\special{pa 864 788}%
\special{pa 864 1164}%
\special{fp}%
\special{sh 1}%
\special{pa 864 1164}%
\special{pa 884 1096}%
\special{pa 864 1110}%
\special{pa 844 1096}%
\special{pa 864 1164}%
\special{fp}%
\put(8.6900,-12.1600){\makebox(0,0){$(13)$}}%
\put(8.6900,-12.1600){\makebox(0,0){$(13)$}}%
\put(8.5800,-17.3100){\makebox(0,0){$(6)$}}%
%
\special{pn 13}%
\special{pa 858 1288}%
\special{pa 858 1664}%
\special{fp}%
\special{sh 1}%
\special{pa 858 1664}%
\special{pa 878 1596}%
\special{pa 858 1610}%
\special{pa 838 1596}%
\special{pa 858 1664}%
\special{fp}%
\put(14.7800,-17.2800){\makebox(0,0){$(7)$}}%
\put(20.8100,-17.3100){\makebox(0,0){$(8)$}}%
\put(27.3200,-17.3100){\makebox(0,0){$(9)$}}%
\put(27.3200,-17.3100){\makebox(0,0){$(9)$}}%
\put(33.9400,-17.3100){\makebox(0,0){$(10)$}}%
\put(40.2400,-17.2900){\makebox(0,0){$(11)$}}%
\put(40.2400,-17.2900){\makebox(0,0){$(11)$}}%
\put(46.2800,-17.3400){\makebox(0,0){$(12)$}}%
%
\special{pn 13}%
\special{pa 1026 1294}%
\special{pa 2622 1668}%
\special{fp}%
\special{sh 1}%
\special{pa 2622 1668}%
\special{pa 2562 1634}%
\special{pa 2570 1656}%
\special{pa 2554 1672}%
\special{pa 2622 1668}%
\special{fp}%
\put(33.8900,-7.2500){\makebox(0,0){$(18)$}}%
\put(21.1200,-12.0600){\makebox(0,0){$(14)$}}%
\put(33.9400,-12.0600){\makebox(0,0){$(15)$}}%
%
\special{pn 13}%
\special{pa 3394 788}%
\special{pa 3390 1158}%
\special{fp}%
\special{sh 1}%
\special{pa 3390 1158}%
\special{pa 3410 1092}%
\special{pa 3390 1106}%
\special{pa 3370 1092}%
\special{pa 3390 1158}%
\special{fp}%
\special{pa 3390 1158}%
\special{pa 3390 1158}%
\special{fp}%
\put(48.3200,-12.1200){\makebox(0,0){$(16)$}}%
%
\special{pn 13}%
\special{pa 3310 788}%
\special{pa 2154 1160}%
\special{fp}%
\special{sh 1}%
\special{pa 2154 1160}%
\special{pa 2224 1158}%
\special{pa 2206 1144}%
\special{pa 2212 1120}%
\special{pa 2154 1160}%
\special{fp}%
%
\special{pn 13}%
\special{pa 3526 790}%
\special{pa 4828 1156}%
\special{fp}%
\special{sh 1}%
\special{pa 4828 1156}%
\special{pa 4770 1118}%
\special{pa 4778 1142}%
\special{pa 4758 1156}%
\special{pa 4828 1156}%
\special{fp}%
\special{pa 4744 1124}%
\special{pa 4744 1124}%
\special{fp}%
%
\special{pn 13}%
\special{pa 4828 1290}%
\special{pa 3526 1660}%
\special{fp}%
\special{sh 1}%
\special{pa 3526 1660}%
\special{pa 3596 1660}%
\special{pa 3578 1644}%
\special{pa 3586 1622}%
\special{pa 3526 1660}%
\special{fp}%
%
\special{pn 13}%
\special{pa 3404 1290}%
\special{pa 3400 1672}%
\special{fp}%
\special{sh 1}%
\special{pa 3400 1672}%
\special{pa 3422 1606}%
\special{pa 3402 1618}%
\special{pa 3382 1604}%
\special{pa 3400 1672}%
\special{fp}%
\special{pa 2160 1282}%
\special{pa 2748 1656}%
\special{fp}%
\special{sh 1}%
\special{pa 2748 1656}%
\special{pa 2702 1604}%
\special{pa 2704 1628}%
\special{pa 2682 1638}%
\special{pa 2748 1656}%
\special{fp}%
\special{pa 3300 1288}%
\special{pa 2848 1668}%
\special{fp}%
\special{sh 1}%
\special{pa 2848 1668}%
\special{pa 2912 1640}%
\special{pa 2890 1634}%
\special{pa 2886 1610}%
\special{pa 2848 1668}%
\special{fp}%
\special{pa 3210 1284}%
\special{pa 2108 1664}%
\special{fp}%
\special{sh 1}%
\special{pa 2108 1664}%
\special{pa 2178 1660}%
\special{pa 2158 1646}%
\special{pa 2166 1622}%
\special{pa 2108 1664}%
\special{fp}%
%
\special{pn 13}%
\special{pa 2014 1284}%
\special{pa 1482 1666}%
\special{fp}%
\special{sh 1}%
\special{pa 1482 1666}%
\special{pa 1548 1644}%
\special{pa 1526 1636}%
\special{pa 1524 1612}%
\special{pa 1482 1666}%
\special{fp}%
\special{pa 1850 1984}%
\special{pa 1850 1984}%
\special{fp}%
%
\special{pn 13}%
\special{pa 3536 1288}%
\special{pa 4612 1674}%
\special{fp}%
\special{sh 1}%
\special{pa 4612 1674}%
\special{pa 4556 1632}%
\special{pa 4562 1656}%
\special{pa 4542 1670}%
\special{pa 4612 1674}%
\special{fp}%
\put(27.6200,-22.5000){\makebox(0,0){$(2)$}}%
\put(34.0000,-22.5300){\makebox(0,0){$(3)$}}%
%
\special{pn 13}%
\special{pa 2732 1804}%
\special{pa 2732 2180}%
\special{fp}%
\special{sh 1}%
\special{pa 2732 2180}%
\special{pa 2752 2112}%
\special{pa 2732 2126}%
\special{pa 2712 2112}%
\special{pa 2732 2180}%
\special{fp}%
%
\special{pn 13}%
\special{pa 3394 1808}%
\special{pa 3394 2184}%
\special{fp}%
\special{sh 1}%
\special{pa 3394 2184}%
\special{pa 3414 2118}%
\special{pa 3394 2132}%
\special{pa 3374 2118}%
\special{pa 3394 2184}%
\special{fp}%
%
\special{pn 13}%
\special{pa 4020 1806}%
\special{pa 4020 2182}%
\special{fp}%
\special{sh 1}%
\special{pa 4020 2182}%
\special{pa 4040 2116}%
\special{pa 4020 2130}%
\special{pa 4000 2116}%
\special{pa 4020 2182}%
\special{fp}%
%
\special{pn 8}%
\special{pa 4650 1806}%
\special{pa 4650 2182}%
\special{dt 0.045}%
\special{sh 1}%
\special{pa 4650 2182}%
\special{pa 4670 2116}%
\special{pa 4650 2130}%
\special{pa 4630 2116}%
\special{pa 4650 2182}%
\special{fp}%
\put(40.1900,-22.5300){\makebox(0,0){$(4)$}}%
\put(46.4900,-22.5300){\makebox(0,0){$(5)$}}%
%
\special{pn 13}%
\special{pa 2260 1808}%
\special{pa 4370 2180}%
\special{fp}%
\special{sh 1}%
\special{pa 4370 2180}%
\special{pa 4308 2148}%
\special{pa 4318 2170}%
\special{pa 4302 2188}%
\special{pa 4370 2180}%
\special{fp}%
%
\special{pn 13}%
\special{pa 3558 1804}%
\special{pa 4506 2182}%
\special{fp}%
\special{sh 1}%
\special{pa 4506 2182}%
\special{pa 4452 2140}%
\special{pa 4456 2162}%
\special{pa 4438 2176}%
\special{pa 4506 2182}%
\special{fp}%
%
\special{pn 13}%
\special{pa 4534 1808}%
\special{pa 3546 2190}%
\special{fp}%
\special{sh 1}%
\special{pa 3546 2190}%
\special{pa 3616 2184}%
\special{pa 3596 2170}%
\special{pa 3602 2146}%
\special{pa 3546 2190}%
\special{fp}%
\special{pa 3894 1804}%
\special{pa 3468 2184}%
\special{fp}%
\special{sh 1}%
\special{pa 3468 2184}%
\special{pa 3530 2154}%
\special{pa 3508 2148}%
\special{pa 3504 2126}%
\special{pa 3468 2184}%
\special{fp}%
\special{pa 4464 1800}%
\special{pa 2932 2180}%
\special{fp}%
\special{sh 1}%
\special{pa 2932 2180}%
\special{pa 3002 2182}%
\special{pa 2984 2166}%
\special{pa 2992 2144}%
\special{pa 2932 2180}%
\special{fp}%
\put(40.1900,-27.5400){\makebox(0,0){$(1)$}}%
%
\special{pn 13}%
\special{pa 4650 2322}%
\special{pa 4082 2700}%
\special{fp}%
\special{sh 1}%
\special{pa 4082 2700}%
\special{pa 4150 2680}%
\special{pa 4126 2670}%
\special{pa 4126 2646}%
\special{pa 4082 2700}%
\special{fp}%
%
\special{pn 8}%
\special{pa 3426 2314}%
\special{pa 3966 2706}%
\special{dt 0.045}%
\special{sh 1}%
\special{pa 3966 2706}%
\special{pa 3924 2650}%
\special{pa 3924 2674}%
\special{pa 3900 2682}%
\special{pa 3966 2706}%
\special{fp}%
\end{picture}%
\]
\caption{Every arrow means the equivariant blowing-up along a $T$-invariant 
curve and every dotted arrow means
the equivariant  blowing-up along a $T$-invariant point.}\label{pic18a} 
\end{figure}
%
\end{theorem}

The corresponding eighteen doubly $\Z$-weighted triangulations of $S^2$  
are given in Figure \ref{pic18}. A segment attached to an oval 
corresponds to the $T$-invariant smooth 
curve of the center of a blowing-up 
appearing in Figure \ref{pic18a}.
The number, like "(5)" in Figure \ref{pic18}(1), 
near the oval is the number of the Fano $3$-fold obtained 
from the blowing-up. 

For instance, the oval in Figure \ref{pic18}(1) 
means that if we blow up along the curve 
corresponding to the segment with the oval, we obtain the Fano $3$-fold 
in (5). Of course, in this case, by symmetry 
we can choose any other segments, or any other 
$T$-invariant smooth  curves,
 as the blowing-up center. 

A dark gray small circle at a vertex corresponds
 to the exceptional $T$-invariant 
divisor of a blowing-down 
appearing in Figure \ref{pic18a}.
The number, like "(1)" in Figure \ref{pic18}(5), 
near the small circle is the number of the Fano $3$-fold obtained 
from the blowing-down.

For instance, the small circle in 
Figure \ref{pic18}(5) means that if we blow down the $T$-invariant 
divisor corresponding to the vertex with the small circle, 
we obtain the Fano $3$-fold in (1).

In Figure \ref{pic18}, 
we do not indicate the point of a blowing-up center, or 
the exceptional divisor of a blowing-down to a point,
since we do not need it afterwards.
\begin{figure}[p]
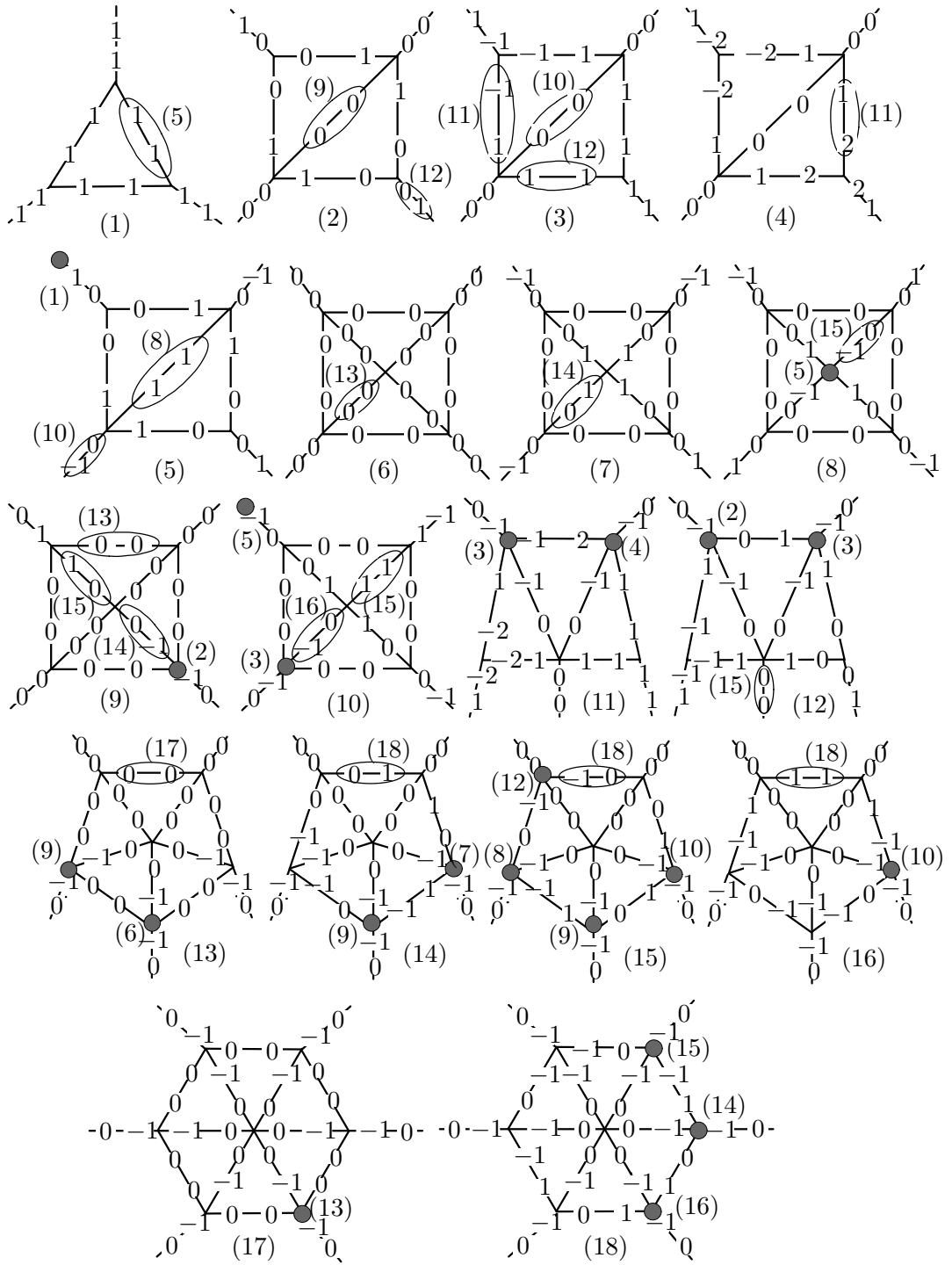

\[
\unitlength 0.1in
%
\]
\caption{Doubly $\Z$-weighted triangulations of $S^2$ 
\cite[page 91]{Od88}}\label{pic18} 
\end{figure}


\subsection{Frobenius push-forward}\label{sub:Frob}
In this subsection, we explain how to compute the direct summands of 
Frobenius push-forward of line bundles on smooth 
complete toric varieties, following 
Thomsen \cite{Th00}. 
 
Fix a positive integer $m$ and
we define a new lattice $N'$
as $N':=\frac{1}{m}N$ and denote its dual by $M'$,
We consider the natural inclusion 
$f_m\colon N \hookrightarrow N'$, which 
sends a cone in $N_{\R}$ to the cone with the same support on $N'_{\R}$.
Thus $f_{m}$ induces
the finite surjective toric morphism 
$F_{ m}\colon X(\Delta)\to X(\Delta)$ which we call a
\emph{Frobenius} (or \emph{multiplication}) \emph{map}.
We put 
$$
\mc V(\Delta)=\{ \mbi v_1,\ldots,\mbi v_l\},
$$
and $D_i$ to be the $T$-invariant 
divisor corresponding to
 the $1$-dimensional cone generated by $\mbi v_i$.
Henceforth, without otherwise specified,
we always assume that $\Delta$ is a complete smooth fan, i.e. 
$X=X(\Delta)$ is a smooth complete toric variety.
We put
$$
A=\t(\mbi v_1,\ldots ,\mbi v_l)\in M(l,n).
$$
 
If $D=\sum_{j=1}^lb_jD_j$ is a $\Q$-divisor, we define
$$\lceil D\rceil :=\sum_{j=1}^l\lceil b_i\rceil D_j,$$
where for any real number $x$, $\lceil x\rceil$ is the integer
defined by $x\leq\lceil x\rceil < x+1$.
Similarly, we define
$$
\lfloor D\rfloor:=\sum_{j=1}^l\lfloor b_j\rfloor D_j,
$$
where for every $x$, $\lfloor x\rfloor$ is the integer defined by
$x-1<\lfloor x\rfloor\leq x$.
$K_X$ denotes the canonical divisor $-\sum_{j=1}^lD_j$
so that $\omega_X=\mc O_X(K_X)$.

For a maximal cone 
$$
\sigma=\Span{\mbi v_{i_1},\ldots,\mbi v_{i_n}}\in \Delta \quad (i_1<\cdots <i_n)
$$
and a matrix
$$
B =\t(\mbi b_1,\ldots, \mbi b_l) \in M(l,m) \quad (l\geq i_n, m\ge 1),
$$
we define
$$
B_\sigma =\t(\mbi b_{i_1},\ldots ,\mbi b_{i_n})\in M(n,m).
$$ 

Then note that for a maximal cone $\sigma$ and the matrix $A$ defined above,
$A_\sigma$ belongs to $\GL (n,\Z)$, since $X$ is a smooth toric variety.

Put
$$
P_m^p=\{ \t(u_1,\ldots ,u_p)\in \Z^p \bigm| 0 \leqq u_i <m \}
$$
for a positive integer $p$. 
For 
$
\mbi u\in P_m^n,\quad \mbi w=\t (w_1,\ldots ,w_l)\in\Z^l
$
and a maximal cone $\sigma\in \Delta$,
define 
$\mbi q^m(\mbi u,\mbi w,\sigma)\in \Z^l,\mbi r^m(\mbi u,\mbi w,\sigma)
\in P_m^l$ and 
$q_i^m(\mbi u,\mbi w,\sigma)\in \Z$
as
\begin{equation}\label{eqn:AA}
AA_{\sigma}^{-1}(\mbi u-\mbi w_\sigma)+\mbi w
=m\mbi q^m(\mbi u,\mbi w,\sigma)+\mbi r^m(\mbi u,\mbi w,\sigma)
\end{equation}
and
\begin{align*}
&\mbi q^m(\mbi u,\mbi w,\sigma)
=\t (q_1^m(\mbi u,\mbi w,\sigma),\ldots ,q_l^m(\mbi u,\mbi w,\sigma)).
\end{align*}
Define 
$$
D_{\mbi u,\mbi w,\sigma}(=D_{\mbi u,\mbi w,\sigma}^m)
:=\sum q_i^m(\mbi u,\mbi w,\sigma)D_i.
$$


\begin{rem}
(i) Suppose that 
$$
\mbi a-\mbi b=A\mbi u
$$
for $\mbi a,\mbi b\in \Z^ l$ and $\mbi u\in \Z ^n$. 
Then we know that 
$$
\sum_{i=1}^l a_iD_i-\sum_{i=1}^l b_iD_i=\Div \chi ^{\mbi u}.
$$
In particular the divisors $\sum_{i=1}^l a_iD_i$ and 
$\sum_{i=1}^l b_iD_i$ are linearly equivalent.

(ii)
We have 
$\Div \chi^{A_\sigma^{-1}\mbi q}|_{U_\sigma}
=\sum_{i=1}^n q_iD_i|_{U_\sigma}$ 
for any 
$\mbi q= \t(q_1,\ldots, q_n) \in \Z^n$.
\end{rem}


\begin{exa}\label{exa:affine}
Put $R=R_\sigma$ for an $n$-dimensional non-singular strongly convex rational 
cone $\sigma$ in $N$. 
For a smooth affine toric variety $U=\Spec R$,
the multiplication map $F_m$ induces a $\C$-algebra map  
$$
F_m^{\#}\colon R \to R \quad \chi ^{\mbi u}\mapsto \chi ^{m\mbi u}.
$$ 
When we regard a quotient field $K$ of $R$ 
as an $R$-module via the map $F_m^{\#}$, we denote it by $F_{m*}K$.
For a sub-$R$-module $L$ of $K$, we also define a sub-$R$-module
$F_{m*}L$ of $F_{m*}K$, which is just $L$ as an abelian group. 

Then the module $F_{m*}(R\chi ^{-A_\sigma^{-1} \mbi w})$ for some $\mbi w \in \Z^n$
is freely generated by the set 
$$
\bigl \{ \chi^{A_\sigma^{-1}(\mbi u-\mbi w)}\bigm| \mbi u\in P_m^n \bigr\},
$$
namely, we have an isomorphism
$$
F_{m*}(R\chi ^{-A_\sigma^{-1} \mbi w})\cong 
\bigoplus_{\mbi u\in P_m^n} R\chi^{A_\sigma^{-1}(\mbi u-\mbi w)},
$$
since we also have the following description; 
$$
R=k\bigl[\chi ^{A_\sigma ^{-1}\mbi e_i} \bigm | i=1,\ldots, n  \bigr],
$$
where $\bigl \{\mbi e_1,\ldots, \mbi e_n \bigr \}$ is the standard basis of $M$.
\end{exa}

The isomorphism on an affine piece in Example \ref{exa:affine} 
can be globalized as follows in (ii).


\begin{lem}\label{lem:FmO}
Fix a vector
$
\mbi w=\t (w_1,\ldots ,w_l)\in\Z^l
$
and a maximal cone $\sigma\in \Delta$.
\begin{enumerate}
\item{\cite{Th00}}
The vector bundle 
$$
\bigoplus _{\mbi u\in P_m^n}\mc O_X(D_{\mbi u,\mbi w,\sigma})
$$
does not depend on the choice of a maximal cone $\sigma$.
\item{\cite{Th00}}
We have
$$
\bigoplus _{\mbi u\in P_m^n}\mc O_X(D_{\mbi u,\mbi w,\sigma})\cong
F_{m*}\mc O_X(\sum w_iD_i).
$$
\item 
For a line bundle $\mc L\in \Pic X$, we have
$$(F_{m*}\mc L)^\vee
 \cong F_{m*}(\mc L^\vee\otimes \omega_X^{1-m})
 \cong F_{m*}(\mc L^\vee\otimes \omega_X)\otimes \omega_X ^{-1}. 
$$
\end{enumerate}
\end{lem} 

\begin{proof}
(iii) The second isomorphism follows from the projection formula.
The first one is a direct consequence of the 
Grothendieck--Verdier duality (cf.~\cite[page 86]{Hu06}), but
we give another proof by the use of the above result. 
Put $\mc L=\mc O_X(\sum w_iD_i)$. We have 
$\mc L^\vee\otimes \omega_X^{1-m}=\mc O_X(\sum (m-1-w_i)D_i)$.
Put $\mbi u'=(m-1)\t (1, 1, \ldots , 1)$. Then, for all $\mbi u\in P_m^n$, we can see
\begin{align*}
 q_i(\mbi u'_{\sigma_0}-\mbi u,\mbi u'-\mbi w,\sigma _0) 
=&\lfl \frac{\t\mbi v_i(\mbi u'_{\sigma_0}-\mbi u-\mbi u'_{\sigma_0}
+ \mbi w_{\sigma_0})-w_i+m-1}{m}\rfl \\
=&\lfl -\frac{\t\mbi v_i(\mbi u-\mbi w_{\sigma_0})+w_i+1}{m} \rfl +1\\
=&-\lce \frac{\t\mbi v_i(\mbi u-\mbi w_{\sigma_0})+w_i}{m}+\frac{1}{m} \rce+1\\
=&- q_i(\mbi u,\mbi w,\sigma _0).
\end{align*}
The last equality holds because, in general, the equality  
$\lce \frac{k}{m}+\frac{1}{m}\rce - \lfl\frac{k}{m}\rfl=1$ is true for $k\in \Z$.
This gives the first isomorphism by (iii).
\end{proof}

Below for simplicity, we often identify two isomorphic line bundles. 
For a $T$-invariant divisor $D$ and an integer $m>0$,
we define sets of (isomorphism classes of) line bundles;
$$
\mk D(D)_m:=\{\mc L \in \Pic X \bigm| 
\mc L \mbox{ is a direct summand of $F_{m*}\mc O_X(D)$ } \}
$$
and 
$$
\mk D(D):=\cup _{m>0}\mk D(D)_m.
$$

\paragraph{Convention}
\begin{enumerate}
\item
We may assume that $\mbi v_{1},\ldots,\mbi v_{n}$ forms a standard basis of $\Z ^n$
and put
$\sigma_0=\Span{\mbi v_{1},\ldots,\mbi v_{n}}$.
We often omit $\sigma_0$ in the notation as 
$\mbi q^m(\mbi u,\mbi w)(:=\mbi q^m(\mbi u,\mbi w,\sigma_0))$,
$D_{\mbi u,\mbi w}(:=D_{\mbi u,\mbi w,\sigma_0})$ and so on.
\item
For a zero divisor $D=0$ or a zero vector $\mbi w=\mbi 0$, 
we simply denote $\mk D(0)$ by $\mk D$ (or $\mk D _X$ if 
we need to specify the base variety $X$) and 
$\mbi q^m(\mbi u)(=\mbi q^m(\mbi u,\mbi 0))$.
(In fact, as a consequence of Lemma \ref{lem:mgg}(i) and (ii), we have 
$\mk D=\mk D(0)_m$ for a sufficiently divisible integer $m$.) 
\end{enumerate}

We take $\sigma$ in (\ref{eqn:AA}) to be $\sigma_0$, then 
(\ref{eqn:AA}) becomes the following simpler form
\begin{equation}\label{eqn:AA2}
A(\mbi u-\mbi w_{\sigma_0})+\mbi w
=m\mbi q^m(\mbi u,\mbi w)+\mbi r^m(\mbi u,\mbi w),
\end{equation}
and hence we have
\begin{equation}\label{eqn:AA2i}
q^m_i(\mbi u,\mbi w)=\lfl \frac{\t\mbi v_i(\mbi u-\mbi w_{\sigma_0})+w_i}{m} \rfl .
\end{equation}
%


\begin{lem}\label{lem:mgg}
Fix a $T$-invariant divisor $D=\sum w_i D_i$
and put $\mbi w=\t (w_1,\ldots ,w_l)$.
\begin{enumerate}
\item \cite{Th00}
The set $\mk D(D)$ is finite. 
\item
Put $D':=\sum w_i'D_i$, where 
$$
w_i':=
\begin{cases}
0& \mbox{ for } i \mbox{ with } w_i\ge 0 \\
-1& \mbox{ for } i \mbox{ with } w_i <0. \\
\end{cases}
$$ 
Take $m>0$ satisfying $-1\le \frac{w_i}{m} <1$ for any $i$.
Then $\mc O_X(D')\in \mk D(D)_m$. 
Furthermore we have
$\mk D(D')\subset \mk D(D)_m$ for sufficiently divisible integers $m>0$.
\item
We have $\mk D(D)_m\subset \mk D(lD)_{lm}$ for any $l,m\in \Z_{>0}$.
\end{enumerate}
\end{lem}

\begin{proof}
(i) Since the set $\bigl\{\frac{\t\mbi v_i\mbi u}{m}\bigm | \mbi u \in P_m^n, m\in \Z_{>0} \bigr \}$
is bounded and 
$\frac{\t\mbi v_i\mbi w_{\sigma_0}-w_i}{m}\to 0$ as $m\to \infty$,
the set of integers $\bigl\{ q^m_i(\mbi u,\mbi w ) \bigm| \mbi u \in P_m^n, m\in \Z_{>0} \bigr\}$
is finite. 
Consequently, so is the set $\mk D(D)$. 

(ii)
For any $\mbi w_{\sigma_0}\in \Z^n$, there is a vector $\mbi u'\in \Z^n$ such that 
$m\mbi u'+\mbi w_{\sigma_0}\in P_m^n$. Put $\mbi u:=m\mbi u'+\mbi w_{\sigma_0}$.
Then we can see 
\begin{align*}
q^m_i(\mbi u,\mbi w)
=&\lfl \frac{\t\mbi v_i(\mbi u-m\mbi u'-\mbi w_{\sigma_0})+ w_i}{m} \rfl 
+ \frac{m\t\mbi v_i \mbi u'}{m}\\
=&\lfl \frac{w_i}{m} \rfl+\t\mbi v_i \mbi u'
=  w _i'+\t\mbi v_i \mbi u',
\end{align*}
which means the divisor
$D_{\mbi u,\mbi w}$ is linearly equivalent to $D'$. 
Thus $\mc O_X(D')\in \mk D(D)_m$.

Take an element $\mc L\in \mk D(D')$ and 
suppose that $\mc L\in \mk D(D)_m$ for some $m$ 
satisfying  $-1\le \frac{w_i}{m} <1$ for any $i$.
Then $\mc L\in \mk D(D)_{km}$ for all $k>0$, 
since $F_{km*}\mc O_X(D)= F_{m*}F_{k*}\mc O_X(D)$.
This gives the last assertion.

(iii) By definition, we have $q^m_i(\mbi u,\mbi w)=q^{lm}_i(l\mbi u,l\mbi w)$,
which implies the conclusion.
\end{proof}

My optimistic conjecture is as follows;


\begin{conj}\label{conj:Bondal}
Let $X$ be a smooth complete toric variety.
Then $F_{m*}\mc O_X(D)$ is a classical generator of $D^b(X)$ for
any $T$-invariant divisors $D$ and a sufficiently large integer $m$. 
\end{conj}

In order to prove Conjecture \ref{conj:Bondal}, 
by Lemma \ref{lem:mgg}(ii) it is essential to show it for
$T$-invariant divisors $D=\sum w_i D_i$
with $w_i= 0$ or $-1$. 

\begin{rem}\label{rem:Bondal}
Bondal announced in \cite{Bo06} that Conjecture \ref{conj:Bondal} 
is true for the case $D=0$.
Although the proof is not available so far,
several people have already used this statement (cf.~\cite{BT10,CM10,CM,DLM10}).
In this article, we refer this statement as \emph{Bondal's conjecture}.
Bondal's conjecture is solved for $2$-dimensional toric Deligne--Mumford stacks in \cite{OU13}.
\end{rem}

Lemma \ref{lem:mgg} will not be used afterwards, but 
an idea to solve Bondal's conjecture 
in \S \ref{sub:110} and \ref{sub:180} comes from it.

The following is sometimes powerful when we show that $F_{m*}\mc O_X$
is a tilting object.

\begin{lem}\label{lem:vani}
Take a line bundle $\mc L$ on $X$.
\begin{enumerate}
\item If $\mc L^{-1}$ is nef, then we have
$\Ext^i_X(\mc L, F_{m*}\mc O_X)=0$ for $i>0$.
\item\cite{Sa09}
If $\mc L\otimes \omega^{-1}_X$ is ample,
then we have
$\Ext^i_X(F_{m*}\mc O_X,\mc L)=0$ for $i>0$.
\end{enumerate}
\end{lem}

\begin{proof}
(i) By adjunction, we have $\Ext^i_X(\mc L, F_{m*}\mc O_X)=H^i(X, \mc L^{-m})$.
But the last term vanishes for $i>0$, since $X$ is toric. 

(ii) We have $\Ext^i_X(F_{m*}\mc O_X,\mc L)
=H^i(X,F_m^*(\mc L\otimes \omega^{-1}_X)\otimes \omega_X)$, which vanishes 
by the Kodaira vanishing theorem.
\end{proof}

We have the following easy lemma.
Because of it, the facts that the set $\mk D_X$ 
forms a full strong exceptional collection and
that 
$F_{m*}\mc O_X$ is a tilting generator for sufficiently large $m$
are equivalent.


\begin{lem}\label{lem:tilting&exceptional}
\begin{enumerate}
\item
Let us consider a finite set of line bundles $\{\mc L_k\}$ satisfying
$\mc L_i\not \cong \mc L_j$ for $i\ne j$. 
Assume that the vector bundle $\mc E=\bigoplus \mc L_k$ 
is a \emph{tilting generator} of $D^b(X)$, namely 
it satisfies the following conditions:
\begin{enumerate}
\item 
$\Hom_X^i(\mc E,\mc E)=0$ for $i\ne 0$. Such an object $\mc E$ in 
$D^b(X)$ is called a \emph{tilting} object.
\item 
$\Span{\mc E}^{\perp}=0$ in $D^b(X)$, that is, $\mc E$ is a 
\emph{generator} of $D^b(X)$.
\end{enumerate}
Then the set $\{\mc L_i\}$
forms a full strong exceptional collection.
\item
Suppose that we have a full strong exceptional collection 
$\{\mc L_k\}$ on $X$.
Then their direct sum $\bigoplus _k\mc L_k$ is a tilting generator of $D^b(X)$.
\end{enumerate}
\end{lem}

\begin{proof} 
The most parts of the statements are direct consequences of the definitions.
I explain only how to show the fullness in (i).

Note that the condition (1) implies that the set $\{\mc L_k\}$
is a strong exceptional collection.
Then we have a semi-orthogonal decomposition (cf.~\cite[page 25]{Hu06}) of 
$D^b(X)$ into $\Span{\{\mc L_k\}}^{\perp}$ and 
$\Span{\{\mc L_k\}}$.
Since $\mc E$ is a generator, $\Span{\{\mc L_k\}}^{\perp}=0$
which implies that the strong exceptional collection $\{\mc L_k\}$ is full. 
\end{proof}


\section{Examples}\label{section:Examples}
In this section, we determine the set $\mk D$ for various 
smooth toric varieties.

\subsection{Maximal toric del Pezzo surface}\label{sub:del}
Let us consider the toric surface $X=Y_3$ which is obtained by 
the blow up of $\PP ^2$ at the three $T$-invariant points. Namely,
$X$ is the \emph{maximal} toric del Pezzo surface with respect to 
birational relations.
 We put
\begin{align*}
\mbi v_1&=
\left( \begin{array}{c}
1\\ 0
\end{array} \right),
\mbi v_2=
\left( \begin{array}{c}
0\\ 1
\end{array}\right),
\mbi v_3=
\left( \begin{array}{c}
-1\\ 1
\end{array}\right),\\
\mbi v_4&=
\left( \begin{array}{c}
-1\\ 0
\end{array}\right),
\mbi v_5=
\left( \begin{array}{c}
0\\ -1
\end{array}\right),
\mbi v_6=
\left( \begin{array}{c}
1\\ -1
\end{array}\right)
\in \Z ^2.
\end{align*}
Then we know that $D_2,D_4,D_6$ are exceptional divisors of $X\to \PP ^2$. 
 Note that $D_1+D_6 \sim D_3+D_4$ and $D_2+D_3\sim D_5+D_6$.  

For 
$
\mbi u=
\left( \begin{array}{c}
x\\ y
\end{array}
\right)\in P_m^2,
$
we have
$$
\mbi q^m(\mbi u)=
\left( 
\begin{array}{c}
\lfl\frac{x}{m}\rfl\\ 
\lfl\frac{y}{m}\rfl\\
\lfl\frac{-x+y}{m}\rfl\\  
\lfl\frac{-x}{m}\rfl\\
\lfl\frac{-y}{m}\rfl\\
\lfl\frac{x-y}{m}\rfl\\
\end{array} \right)
=
\left( \begin{array}{c}
0\\ 0\\
\lfl\frac{-x+y}{m}\rfl\\  
\lfl\frac{-x}{m}\rfl\\
\lfl\frac{-y}{m}\rfl\\
\lfl\frac{x-y}{m}\rfl\\
\end{array} \right).
$$
Then we obtain
\begin{align*}
\mk D=\bigl\{ & \mc O_X(-D_5-D_6), \mc O_X(-D_3-D_4), \mc O_X(-D_4-D_5),\\
&\mc O_X(-D_3-D_4-D_5),\mc O_X(-D_4-D_5-D_6),\mc O_X 
\bigr\}.
\end{align*}
These are dual to the line bundles which appear in a
 full strong exceptional collections on $X$ in \cite{Ki}.
In particular, $\mk D$ forms a full strong exceptional collection.

\subsection{Fano $3$-fold in (11)}\label{sub:110}
Take the Fano $3$-fold $X$ in (11). Put
\begin{align*}
\mbi v_1&=
\left( 
\begin{array}{cc}
1\\ 0\\0
\end{array} \right),
\mbi v_2=
\left( \begin{array}{cc}
0\\ 1\\0
\end{array}\right),
\mbi v_3=
\left( \begin{array}{cc}
0\\ 0\\ 1
\end{array}\right),\\
\mbi v_4&=
\left( \begin{array}{cc}
1\\ 0\\-1
\end{array}\right),
\mbi v_5=
\left( \begin{array}{cc}
0\\ 0\\-1
\end{array}\right),
\mbi v_6=
\left( \begin{array}{c}
-1\\ -1 \\2
\end{array}\right)
\in \Z ^2,
\end{align*}
as in Figure \ref{Fano11&18}.
 
\begin{figure}[t]
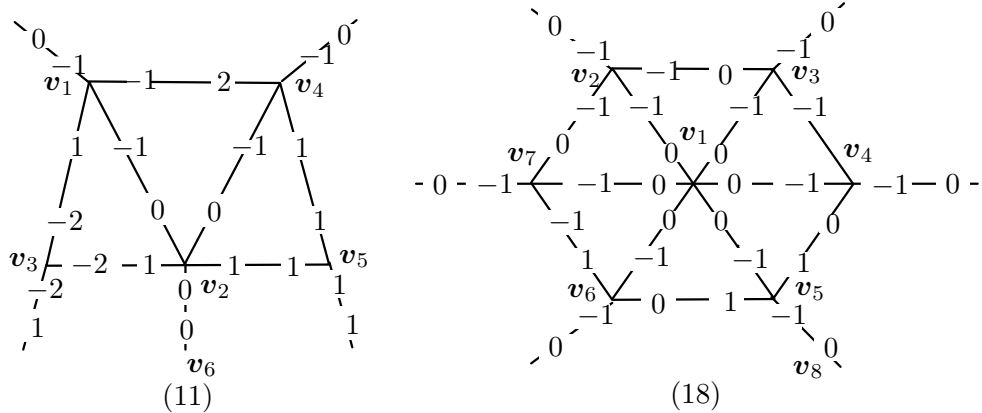

\[
\unitlength 0.1in
%
%
\]
\caption{Fano $3$-folds in (11) and (18)}\label{Fano11&18} 
\end{figure}

For 
$
\mbi u=
\left( \begin{array}{cc}
x\\ y \\ z \\
\end{array}
\right)\in P_m^3,
$
we have
$$
\mbi q^m(\mbi u)=
\left( 
\begin{array}{c}
\lfl\frac{x}{m}\rfl \\
\lfl\frac{y}{m}\rfl \\
\lfl\frac{z}{m}\rfl \\ 
\lfl\frac{x-z}{m}\rfl\\
\lfl\frac{-z}{m}\rfl\\ 
\lfl\frac{-x-y+2z}{m}\rfl\\ 
\end{array} \right)
=
\left( 
\begin{array}{c}
0\\ 0\\ 0\\ 
\lfl\frac{x-z}{m}\rfl\\
\lfl\frac{-z}{m}\rfl\\ 
\lfl\frac{-x-y+2z}{m}\rfl\\ 
\end{array} \right).
$$
Therefore we have
\begin{align*}
\mk D=\bigl\{ &\mc O_X,\mc O_X(-D_6),\mc O_X(-2D_6), \mc O_X(-D_5),
\mc O_X(-D_5-D_6),\mc O_X(-D_5-2D_6),\\
 &\mc O_X(-D_4-D_5-D_6), 
\mc O_X(-D_4-D_5),\mc O_X(-D_4-D_5+D_{6}) \bigr\}.
\end{align*}
Then by the equation (\ref{eqn:inter}) 
we can read from Figure \ref{Fano11&18}
that for all $\mc L\in\mk D$ except $\mc O_X(-D_4-D_5+D_{6})$,
$\mc L^{-1}$ is nef, hence Lemma \ref{lem:vani}
implies that  $\Ext_X^i(\mc L,F_{m*}\mc O_X)=0$ for $i>0$. 
Put 
$$
D_{\nef}:=\mk D\backslash \{\mc O_X(-D_4-D_5+D_{6}) \}.
$$
We shall prove in \S \ref{sub:11} that 
$\Span{\mk D_{\nef}}=D^b(X)$. Consequently the set $\mk D_{\nef}$ becomes 
a full strong exceptional collection. 
 
\subsection{Fano $3$-fold in (18)}\label{sub:180} 
Take the Fano $3$-fold in (18). Put
\begin{align*}
\mbi v_1&=
\left( \begin{array}{cc}
1\\ 0\\0
\end{array} \right),
\mbi v_2=
\left( \begin{array}{cc}
0\\ 1\\0
\end{array}\right),
\mbi v_3=
\left( \begin{array}{cc}
0\\ 0\\ 1
\end{array}\right),
\mbi v_4=
\left( \begin{array}{cc}
 0\\-1\\ 1
\end{array}\right),\\
\mbi v_5&=
\left( \begin{array}{cc}
0\\ -1\\0
\end{array}\right),
\mbi v_6=
\left( \begin{array}{c}
 0\\ 0\\-1
\end{array}\right),
\mbi v_7=
\left( \begin{array}{cc}
0\\ 1\\-1
\end{array}\right)
\mbi v_8=
\left( \begin{array}{cc}
-1\\ 1\\ 0
\end{array}\right)
\in \Z ^2,
\end{align*}
as in Figure \ref{Fano11&18}.
For 
$
\mbi u=
\left( \begin{array}{cc}
x\\ y \\ z \\
\end{array}
\right)\in P_m^3,
$
we have
$$
\mbi q^m(\mbi u)=
\left( 
\begin{array}{c}
\lfl\frac{x}{m}\rfl \\
\lfl\frac{y}{m}\rfl \\
\lfl\frac{z}{m}\rfl \\ 
\lfl\frac{-y+z}{m}\rfl\\
\lfl\frac{-y}{m}\rfl\\ 
\lfl\frac{-z}{m}\rfl\\ 
\lfl\frac{y-z}{m}\rfl\\
\lfl\frac{-x+y}{m}\rfl\\ 
\end{array} \right)
=
\left( 
\begin{array}{c}
0\\ 0\\ 0\\ 
\lfl\frac{-y+z}{m}\rfl\\
\lfl\frac{-y}{m}\rfl\\ 
\lfl\frac{-z}{m}\rfl\\ 
\lfl\frac{y-z}{m}\rfl\\
\lfl\frac{-x+y}{m}\rfl\\ 
\end{array} \right).
$$
Therefore we have
\begin{align*}
\mk D=\bigl \{ &\mc O_X(-iD_8), \mc O_X(-D_6-D_7-iD_8),
\mc O_X(-D_4-D_5-iD_{8}),\\ 
&\mc O_X(-D_5-D_6-D_{7}-iD_8), \mc O_X(-D_5-D_{6}-iD_8), \\
&\mc O_X(-D_4-D_5-D_{6}-iD_8) \bigm| i=0,1 
 \bigr \}.
 \end{align*}
By the equation (\ref{eqn:inter}),
 we can read from Figure \ref{Fano11&18}
that that $\mc L^{-1}$ is nef for all $\mc L\in \mk D$. Hence by 
Lemma \ref{lem:vani} implies that 
$\Ext^i _X(F_{m*}\mc O_X,F_{m*}\mc O_X)=0$ for all $m\gg 0,i>0$.  


\subsection{Fano $3$-fold in (8)}
Take the Fano $3$-fold $X$ in (8). Put
\begin{align*}
\mbi v_1&=
\left( \begin{array}{cc}
1\\ 0\\0
\end{array} \right),
\mbi v_2=
\left( \begin{array}{cc}
0\\ 1\\0
\end{array}\right),
\mbi v_3=
\left( \begin{array}{cc}
0\\ 0\\ 1
\end{array}\right),\\
\mbi v_4&=
\left( \begin{array}{cc}
-1\\ 0\\-1
\end{array}\right),
\mbi v_5=
\left( \begin{array}{cc}
1\\ -1\\0
\end{array}\right),
\mbi v_6=
\left( \begin{array}{c}
-1\\ 0\\0
\end{array}\right)
\in \Z ^2.
\end{align*}

For 
$
\mbi u=
\left( \begin{array}{cc}
x\\ y \\ z \\
\end{array}
\right)\in P_m^3,
$
we have
$$
\mbi q^m(\mbi u)=
\left( 
\begin{array}{c}
\lfl\frac{x}{m}\rfl \\
\lfl\frac{y}{m}\rfl \\
\lfl\frac{z}{m}\rfl \\ 
\lfl\frac{-x-z}{m}\rfl\\
\lfl\frac{x-y}{m}\rfl\\ 
\lfl\frac{-x}{m}\rfl\\ 
\end{array} \right)
=
\left( 
\begin{array}{c}
0\\ 0\\ 0\\ 
\lfl\frac{-x-z}{m}\rfl\\
\lfl\frac{x-y}{m}\rfl\\ 
\lfl\frac{-x}{m}\rfl\\  
\end{array} \right).
$$
Therefore we have
\begin{align*}
\mk D=\bigl\{ &\mc O_X,\mc O_X(-D_5),\mc O_X(-D_4), \mc O_X(-D_4-D_5),
 \mc O_X(-D_4-D_6),\\
&\mc O_X(-D_4-D_5-D_{6}), \mc O_X(-2D_4-D_6),\mc O_X(-2D_4-D_5-D_{6}) \bigr\}.
 \end{align*}
For all of line bundles $\mc L\in \mk D$, we can see that
$\mc L^{-1}$ is nef by a similar method to one above, hence
Lemma \ref{lem:vani}
implies that  $\Ext_X^i(F_{m*}\mc O_X,F_{m*}\mc O_X)=0$ for $i>0$. 
This result contradicts the result in \cite[page 32]{Sa09}.


\section{Exceptional collections on maximal toric Fano $3$-folds}
\label{sec:Bondal}
In this section, 
we prove Bondal's conjecture 
for maximal smooth toric Fano $3$-folds. Combining this with the results in 
\S \ref{section:Examples},
we see that $\mk D_X$ (respectively, $\mk D_{\nef}$) is a 
full strong exceptional collection in the cases Fano $3$-folds in
(17) and (18) (respectively, (11)).
 
\begin{lem}\label{lem:f_*}
Let $f\colon X\to Y$ be a proper morphism between smooth varieties.
Suppose that  an object $\mc E$ is a generator of $D^b(X)$ and $\mc O_Y$ is a direct summand of the 
object $\mb Rf_*\mc O_X$.
Then $\mb R f_*\mc E$ is also a generator of $D^b(Y)$.
\end{lem}

\begin{proof}
Put $\RHom _Y(\mb R f_*\mc E, \mc F)=0$ for some $\mc F\in D^b(Y)$. 
Then by the adjointness $\mb Rf_* \dashv f^!$,
we obtain $\omega _X\otimes \mb Lf^*(\mc F \otimes \omega_Y^{-1})=f^!\mc F=0$,
which implies that $\mb Lf^*\mc F=0$. Since $\mc F$ is a direct summand of 
$\mb R f_*\mb Lf^*\mc F=\mc F\Lotimes \mb Rf_*\mc O_X$, we obtain the assertion. 
\end{proof}

In Lemma \ref{lem:f_*}, main examples in mind are the following:
Let $X$ and $Y$ be smooth projective toric varieties.
 
(i) For the toric blow up $f\colon X\to Y$,
by $\mb R f_*\mc O_X= \mc O_Y$ and the commutativity $F_{m}^Y\circ f=f\circ F_m^X$,
we have $\mb R f_*F_{m*}^X\mc O_X=F_{m*}^Y\mc O_Y$. Hence if $F_{m*}O_X$ is a generator, then so is 
$F_{m*}\mc O_Y$.

(ii) For the Frobenius morphism $F_m$ on $X$, $F_{m*}\mc E$ is a generator of $D^b(X)$ 
for a generator $\mc E$ of $D^b(X)$. Note that $\mc O_X$ is indeed a direct summand of 
$F_{m*}\mc O_X$.


\subsection{Exceptional collection on the Fano $3$-fold in (17)}\label{sub:17}

Let us consider the Fano $3$-fold $X$ in (17) which is the product of the maximal toric del Pezzo surface $Y_3$ 
in \S \ref{sub:del} and a projective line $\PP ^1$. The following must be well-known.


\begin{lem}\label{lem:product}
Let $Y$ and $Z$ be smooth projective varieties.
Suppose that $\mc E$ and $\mc F$ 
are tilting generators 
of $D^b(Y)$ and $D^b(Z)$ respectively.
Then $\mc E\Lboxtimes \mc F$ is also a tilting generator of 
$D^b(Y\times Z)$.  
\end{lem}

\begin{proof}
We can check 
$$
\RGamma (Y\times Z,\mc E\Lboxtimes \mc F)
=\RGamma (Y,\mc E)\Lotimes \RGamma (Z, \mc F).
$$
Hence we have 
$$
\Hom ^i_{Y\times Z}(\mc E\Lboxtimes \mc F ,\mc E\Lboxtimes \mc F )\cong 
\bigoplus _{j+k=i}\Hom ^j_Y(\mc E, \mc E)
\otimes \Hom ^k_Z(\mc F, \mc F ),
$$  
which implies that $\mc E\Lboxtimes \mc F$ is tilting. 
The fact $\mc E\Lboxtimes \mc F$ is a generator directly follows from 
\cite[Lemma 3.4.1]{BV03}.
\end{proof}

For the toric case, we know that
$F_{m*}^{Y\times Z}\mc O_{Y\times Z} \cong F_{m*}^Y\mc O_Y \boxtimes F_{m*}^Z\mc O_Z$.
Moreover we see in \S \ref{sub:del} that $F_{m*}\mc O_{Y_3}$ is a tilting generator 
on the maximal toric del Pezzo surface $Y_3$ and $m\gg 0$.
It follows that $\mk D_X$ is a full strong exceptional collection. Here we leave to readers the proof of the fact 
that $\mk D_{\PP ^1}$ is a full strong exceptional collection on $\PP^1$.


\subsection{Exceptional collection on Fano $3$-fold in (11)} \label{sub:11}
Take the Fano $3$-fold $X$ in (11). We use the same notation as in 
\S \ref{sub:110}.
Let us recall that 
\begin{align*}
\mk D_{\nef}=\bigl\{ &\mc O_X,\mc O_X(-D_6),\mc O_X(-2D_6), \mc O_X(-D_5),
\mc O_X(-D_5-D_6),\\ &\mc O_X(-D_5-2D_6),\mc O_X(-D_4-D_5-D_6), 
\mc O_X(-D_4-D_5)\bigr\}.
\end{align*}
The next is the aim of \S \ref{sub:11}.


\begin{prop}\label{prop:(11)}
Let $X$ be the toric Fano $3$-fold in (11).
Then the set $\mk D_{\nef}$ forms a full strong exceptional collection.
\end{prop}

First we determine the set $\mk D(\omega _X^{-3})_m$
for sufficiently large $m$.
For 
$
\mbi u=
\left( \begin{array}{cc}
x\\ y \\ z \\
\end{array}
\right)\in P_m^3
$
and 
$
\mbi w=
\left( \begin{array}{ccc}
-1\\ \vdots \\ -1 \\
\end{array}
\right)\in \Z^6,
$
we have
$$
\mbi q^m(\mbi u,-3 \mbi w)=
\left( 
\begin{array}{c}
\lfl\frac{(x-3)+3}{m}\rfl \\
\lfl\frac{(y-3)+3}{m}\rfl \\
\lfl\frac{(z-3)+3}{m}\rfl \\ 
\lfl\frac{(x-3)-(z-3)+3}{m}\rfl\\
\lfl\frac{-(z-3)+3}{m}\rfl\\ 
\lfl\frac{-(x-3)-(y-3)+2(z-3)+3}{m}\rfl\\ 
\end{array} \right)
=
\left( 
\begin{array}{c}
0\\
0\\
0\\
\lfl\frac{x-z+3}{m}\rfl \\
\lfl\frac{-z+6}{m}\rfl \\
\lfl\frac{-x-y+2z+3}{m}\rfl \\ 
\end{array} \right)
.
$$
Thus we have
\begin{align*}
q^m_4(\mbi u,-3 \mbi w)&=
\begin{cases}
1& \mbox{ if } x+3\geq z+m \\
0& \mbox{ if } z+m> x+3\geq z\\
-1& \mbox{ if } z>x+3
\end{cases}\\
q^m_5(\mbi u,-3 \mbi w)&=
\begin{cases}
0& \mbox{ if } 6\geq z\\
-1& \mbox{ if } z>6 
\end{cases}\\
q^m_6(\mbi u,-3 \mbi w)&=
\begin{cases}
2& \mbox{ if } 2z+3\geq x+y+2m \\
1& \mbox{ if } x+y+2m> 2z+3\geq x+y+m \\
0& \mbox{ if } x+y+m>2z+3\geq x+y \\
-1& \mbox{ if } x+y>2z+3\geq x+y-m \\
-2& \mbox{ if } x+y-m>2z+3.
\end{cases}
\end{align*}
By tedious computation, we can see
$
\mk D (\omega _X^{-3})_m =\mk D \cup \mk D',
$
where 
\begin{align*}
\mk D'=
\bigl \{ &\mc O_X(-D_4-D_5+2D_6),\mc O_X(-D_5+D_6),\\
&\mc O_X(-D_4-iD_6),
\mc O_X(D_4-jD_6 ) \bigm| i=0,1 \mbox{ and } j=1,2 \bigr \}.
\end{align*}
Note that there are linear equivalences;
\begin{equation}\label{eqn:linear11}
D_1+D_4\sim D_6, \quad D_2\sim D_6,  \quad D_3+2D_6\sim D_4+ D_5.
\end{equation}

\begin{cla}\label{Cl:nef}
$\Span{\mk D_{\nef}}=\Span{\mk D(\omega_X^{-3})_m}$.
\end{cla}

\begin{proof}
We shall check that $\mc L\in \Span{\mk D_{\nef}}$ for 
$\mc L=\mc O_X(-D_4-D_5+D_6)$ and all 
$\mc L\in \mk D'$ below. 
Note that this implies that $\Span{\mk D_{\nef}}=\Span{\mk D(\omega_X^{-3})_m}$, since
$$
\mk D(\omega_X^{-3})_m=\mk D\cup \mk D'
=\mk D_{\nef}\cup\{\mc O_X(-D_4-D_5+D_6) \}\cup \mk D'.
$$

Since $D_1\cap D_2\cap D_6=\emptyset$,
we have an exact sequence
\begin{align*}
0\to & \mc O_X(-D_1-D_2-D_6)\\
\to & \mc O_X(-D_1-D_2)\oplus \mc O_X(-D_2-D_6)\oplus \mc O_X(-D_1-D_6)\\
\to &
\mc O_X(-D_1)\oplus \mc O_X(-D_2)\oplus \mc O_X(-D_6)
\to
\mc O_X \to 0.
\end{align*}
Combining this with (\ref{eqn:linear11}), we have an exact sequence
\begin{align}\label{al:126}
0\to& \mc O_X(D_4-3D_6)\to 
\mc O_X(D_4-2D_6)^{\oplus 2}\oplus \mc O_X(-2D_6)\nonumber\\
\to &
\mc O_X(D_4-D_6 )\oplus \mc O_X(-D_6)^{\oplus 2}
\to
\mc O_X \to 0.
\end{align}
Similarly, since $D_2\cap D_4\cap D_6=\emptyset$,
we have an exact sequence
\begin{align}\label{al:246}
0\to& \mc O_X(-D_4-2D_6)\to 
\mc O_X(-D_4-D_6)^{\oplus 2}\oplus \mc O_X(-2D_6)\nonumber\\
\to &
\mc O_X(-D_4)\oplus \mc O_X(-D_6)^{\oplus 2}
\to
\mc O_X \to 0.
\end{align}

Since $D_3\cap D_4=\emptyset$,
we have an exact sequence
\begin{align*}
0\to  \mc O_X(-D_3-D_4)
\to  \mc O_X(-D_3)\oplus \mc O_X(-D_4)
\to 
\mc O_X \to 0.
\end{align*}
Using (\ref{eqn:linear11}), we have an exact sequence
\begin{align}\label{al:34}
0\to& \mc O_X(-2D_4-D_5+2D_6)\nonumber\\
\to &
\mc O_X(-D_4-D_5+2D_6)\oplus \mc O_X(-D_4)
\to \mc O_X \to 0.
\end{align}
Similarly, since $D_1\cap D_5=\emptyset$,
 we have an exact sequence
\begin{align}\label{al:15}
0\to& \mc O_X(D_4-D_5-D_6)\nonumber\\
\to &
\mc O_X(D_4-D_6)\oplus \mc O_X(-D_5)
\to
\mc O_X \to 0.
\end{align}

(i) Tensoring $\mc O_X(-D_4-D_5+D_6)$ with (\ref{al:126}),
we obtain  an exact sequence
\begin{align*}
0\to& \mc O_X(-D_5-2D_6)\to 
\mc O_X(-D_5-D_6)^{\oplus 2}\oplus \mc O_X(-D_4-D_5-D_6)\nonumber\\
\to &
\mc O_X(-D_5)\oplus \mc O_X(-D_4-D_5)^{\oplus 2}
\to
\mc O_X(-D_4-D_5+D_6) \to 0.
\end{align*}
We have already known that 
all line bundles in the sequence except $\mc O_X(-D_4-D_5+D_6)$ 
belong to $\Span{\mk D_{\nef}}$. 
Thus so does $\mc O_X(-D_4-D_5+D_6)$.\footnote{This proves the fact 
$
\Span{\mk D_{\nef}}=\Span{\mk D},
$
which has been already observed in \cite[Proposition 3.2]{BT10}.}

(ii) Tensoring $\mc O_X(-D_5+D_6)$ with (\ref{al:246}),
we obtain an exact sequence
\begin{align*}
0\to& \mc O_X(-D_4-D_5-D_6)\to 
\mc O_X(-D_4-D_5)^{\oplus 2}\oplus \mc O_X(-D_5-D_6)\nonumber\\
\to &
\mc O_X(-D_4-D_5+D_6)\oplus \mc O_X(-D_5)^{\oplus 2}
\to
\mc O_X(-D_5+D_6) \to 0.
\end{align*}
We have already known  from (i) 
that all line bundles in the sequence except $\mc O_X(-D_5+D_6)$ 
belong to $\Span{\mk D_{\nef}}$. Thus so does $\mc O_X(-D_5+D_6)$.

(iii) Tensoring $\mc O_X(-D_4-D_5+2D_6)$ with (\ref{al:126}),
we obtain  an exact sequence
\begin{align*}
0\to& \mc O_X(-D_5-D_6)\to 
\mc O_X(-D_5)^{\oplus 2}\oplus \mc O_X(-D_4-D_5)\nonumber\\
\to &
\mc O_X(-D_4-D_5+D_6)^{\oplus 2}\oplus \mc O_X(-D_5+D_6)
\to
\mc O_X(-D_4-D_5+2D_6) \to 0.
\end{align*}
We have already known from (i) and (ii) 
that all line bundles in the sequence except $\mc O_X(-D_4-D_5+2D_6)$ 
belong to $\Span{\mk D_{\nef}}$. Thus so does $\mc O_X(-D_4-D_5+2D_6)$.

(iv)
Take $j=1,2$. Tensoring $\mc O_X(D_4-jD_6)$ with (\ref{al:34}),
we obtain an exact sequence%
\begin{align*}
0\to& \mc O_X(-D_4-D_5+(2-j)D_6)\nonumber\\
\to &
\mc O_X(-D_5+(2-j)D_6)\oplus \mc O_X(-jD_6)
\to \mc O_X(D_4-jD_6) \to 0.
\end{align*}
We have already known from (i) and (ii)
that all line bundles in the sequence except $\mc O_X(D_4-jD_6)$ 
belong to $\Span{\mk D_{\nef}}$. Thus so does $\mc O_X(D_4-jD_6)$.

(v)
Take $i=0,1$. Tensoring $\mc O_X(-D_4-iD_6)$ with (\ref{al:15}),
we obtain an exact sequence
\begin{align*}
0\to& \mc O_X(-D_5-(i+1)D_6)\nonumber\\
\to &
\mc O_X(-(i+1)D_6)\oplus \mc O_X(-D_4-D_5-iD_6)
\to
\mc O_X(-D_4-iD_6) \to 0.
\end{align*}
We have already known
that all line bundles in the sequence except $\mc O_X(-D_4-iD_6)$ 
belong to $\Span{\mk D_{\nef}}$. Thus so does $\mc O_X(-D_4-iD_6)$.

Therefore we know that
$\Span{\mk D_{\nef}}=\Span{\mk D (\omega _X^{-3})_m}$.
\end{proof}

Now we can prove Proposition \ref{prop:(11)}.

\begin{proof}
We directly see by computation that
$$
\mk D_{\nef} \subset \mk D (\omega _X^{-1})_m\subset \mk D (\omega _X^{-2})_m
\subset \mk D (\omega _X^{-3})_m,
$$
 which is more or less expected by Lemma \ref{lem:mgg}. 
It is also known that $\bigoplus _{i=0}^3\omega _X^{-i}$ is a generator 
(\cite[Lemma 3.2.2]{VdB04}), since $\omega_X^{-1}$ is very ample. 
Lemma \ref{lem:f_*} implies that $\Span{\mk D (\omega _X^{-3})_m}^\perp=0$.
Thus we can see from Claim \ref{Cl:nef} that 
$$
\Span{\mk D_{\nef}}^\perp=0.
$$ 
Combining the result in 
\S \ref{sub:110} with
Lemma \ref{lem:tilting&exceptional}(i), we complete the proof.
\end{proof}

\subsection{Exceptional collection on Fano $3$-fold in (18)}\label{sub:18} 
Take the Fano $3$-fold in (18), and use the same notation as in 
\S \ref{sub:180}. 
Recall that
\begin{align*}
\mk D=\bigl \{ &\mc O_X(-iD_8), \mc O_X(-D_6-D_7-iD_8),
\mc O_X(-D_4-D_5-iD_{8}),\\ 
&\mc O_X(-D_5-D_6-D_{7}-iD_8), \mc O_X(-D_5-D_{6}-iD_8), \\
&\mc O_X(-D_4-D_5-D_{6}-iD_8) \bigm| i=0,1 
 \bigr \}.
\end{align*}
We can prove the following.


\begin{prop}\label{prop:(18)}
Let $X$ be the toric Fano $3$-fold in (18).
Then the set $\mk D$ forms a full strong exceptional collection.
\end{prop}

First we want to find all elements of $\mk D(\omega _X^{-3})_m$
for sufficiently large $m$.
For 
$
\mbi u=
\left( \begin{array}{cc}
x\\ y \\ z \\
\end{array}
\right)\in P_m^3
$
and 
$
\mbi w=
\left( \begin{array}{ccc}
-1\\ \vdots \\ -1 \\
\end{array}
\right)\in \Z^8,
$
we have
$$
\mbi q^m(\mbi u,-3 \mbi w)=
\left( 
\begin{array}{c}
\lfl\frac{(x-3)+3}{m}\rfl \\
\lfl\frac{(y-3)+3}{m}\rfl \\
\lfl\frac{(z-3)+3}{m}\rfl \\ 
\lfl\frac{-(y-3)+(z-3)+3}{m}\rfl\\
\lfl\frac{-(y-3)+3}{m}\rfl\\ 
\lfl\frac{-(z-3)+3}{m}\rfl\\ 
\lfl\frac{(y-3)-(z-3)+3}{m}\rfl\\
\lfl\frac{-(x-3)+(y-3)+3}{m}\rfl\\ 
\end{array} \right)
=
\left( 
\begin{array}{c}
0 \\
0 \\
0 \\ 
\lfl\frac{-y+z+3}{m}\rfl\\
\lfl\frac{-y+6}{m}\rfl\\ 
\lfl\frac{-z+6}{m}\rfl\\ 
\lfl\frac{y-z+3}{m}\rfl\\
\lfl\frac{-x+y+3}{m}\rfl\\ 
\end{array} \right).
$$
Thus we have
\begin{align*}
q^m_4(\mbi u,-3 \mbi w)&=
\begin{cases}
1& \mbox{ if } z\geq y+m-3 \\
0& \mbox{ if } y+m-3>z\geq y-3 \\
-1& \mbox{ if } y>z+3
\end{cases}\\
q^m_5(\mbi u,-3 \mbi w)&=
\begin{cases}
0& \mbox{ if } 6\geq y \\
-1& \mbox{ if } y>6 
\end{cases}\\
q^m_6(\mbi u,-3 \mbi w)&=
\begin{cases}
0& \mbox{ if } 6\geq z \\
-1& \mbox{ if } z>6
\end{cases}\\
q^m_7(\mbi u,-3 \mbi w)&=
\begin{cases}
1& \mbox{ if } y\geq z+m-3 \\
0& \mbox{ if } z+m-3>y\geq z-3 \\
-1& \mbox{ if } z>y+3
\end{cases}\\
q^m_8(\mbi u,-3 \mbi w)&=
\begin{cases}
1& \mbox{ if } y\geq x+m-3 \\
0& \mbox{ if } x+m-3>y\geq x-3 \\
-1& \mbox{ if } x>y+3 .
\end{cases}
\end{align*}
Hence by tedious computation, we can see
$
\mk D (\omega _X^{-3})_m =\mk D \cup \mk D',
$
where 
\begin{align*}
\mk D'=
\bigl \{ &\mc O_X(-D_4-iD_8),\mc O_X(-D_5-iD_8),\mc O_X(-D_6-iD_8),\mc O_X(-D_7-iD_8),\\
&\mc O_X(-D_4-D_5+D_7+iD_8), \mc O_X(D_4-D_6-D_7-iD_8),\\
&\mc O_X(-D_4-D_5+D_8),\mc O_X(-D_5-D_6+D_8),\\
&\mc O_X(-D_4-D_5-D_6+D_8) \bigm| i=0,1 \bigr \}.
\end{align*}

Since $\mk D$ contains several line bundles of the form $\mc L$ and $\mc L\otimes \mc O_X(D_8)$,
Claim \ref{cla:D8} gives that $\mc L\otimes \mc O_X(iD_8)\in \Span{\mk D}$ for all $i\in \Z$. 

\begin{cla}\label{cla:D8}
If $\mc L$ and $\mc L\otimes \mc O_X(D_8) \in \Span{\mk D}$,
then we have 
$\mc L\otimes \mc O_X(iD_8)\in \Span{\mk D}$ for all $i\in \Z$. 
\end{cla}
\begin{proof}
Note that there are linearly equivalences;
\begin{equation}\label{eqn:linear}
D_1\sim D_8, \quad D_2+D_4+D_5\sim D_7+D_8,  \quad D_3+D_4\sim D_6+ D_7.
\end{equation}
We have
$$
0\to \mc O_X(-D_1-D_8)\to \mc O_X(-D_1)\oplus \mc O_X(-D_8)\to \mc O_X\to 0.
$$
Combining this with (\ref{eqn:linear}), we have
$$
0\to \mc O_X(-2D_8)\to \mc O_X(-D_8)\oplus \mc O_X(-D_8)\to \mc O_X\to 0.
$$
By tensoring $\mc L\in \Pic X$, we obtain the claim.
\end{proof}


\begin{cla}\label{cla:18}
$\Span{\mk D}=\Span{\mk D (\omega _X^{-3})_m}$.
\end{cla}

\begin{proof}
We shall check that $\mc L\in \Span{\mk D}$ for all $\mc L\in \mk D'$ below. 
First note that Claim \ref{cla:D8} implies  
the last three line bundles in $\mk D'$ belong to $\Span{\mk D}$. 
We take an arbitrary integer $i\in \Z$ below.

(i)
We have exact sequences:
\begin{align*}
0\to \mc O_X(-D_5-D_6-D_7+iD_8)\to& \mc O_X(-D_6-D_7+iD_8)\\
 &\to \mc O_{D_5}(-D_6+iD_8)\to 0,
\end{align*}
\[
0\to \mc O_X(-D_5-D_6+iD_8)\to \mc O_X(-D_6+iD_8)\to \mc O_{D_5}(-D_6+iD_8)\to 0.
\]
Hence 
$\mc O_X(-D_6+iD_8) \in \Span {\mk D}$, since 
$$
\mc O_X(-D_5-D_6-D_7+iD_8), \mc O_X(-D_5-D_6+iD_8), \mc O_X(-D_6-D_7+iD_8)\in \Span{\mk D}.
$$ 
Similarly we obtain 
$\mc O_X(-D_5+iD_8)\in \Span{\mk D}$.

(ii)
We have exact sequences:
\begin{align*}
0\to \mc O_X(-D_3-D_4-D_5+iD_8)\to& \mc O_X(-D_3-D_4+iD_8)\\
\to& \mc O_{D_5}(-D_4+iD_8)\to 0,
\end{align*}
\[
0\to \mc O_X(-D_4-D_5+iD_8)\to \mc O_X(-D_4+iD_8)\to \mc O_{D_5}(-D_4+iD_8)\to 0.
\]
Since the line bundles
$$
\mc O_X(-D_3-D_4-D_5+iD_8)\cong \mc O_X(-D_5-D_6-D_7+iD_8),$$ 
$$
\mc O_X(-D_3-D_4+iD_8)\cong \mc O_X(-D_6-D_7+iD_8)
$$ 
and
$$\mc O_X(-D_4-D_5+iD_8)$$
belong to $\Span{\mk D}$
by (\ref{eqn:linear}),
we have $\mc O_X(-D_4+iD_8)\in \Span{\mk D}$.

(iii)
We have exact sequences:
\begin{align*}
0\to \mc O_X(-D_2-D_6-D_7+iD_8)\to &\mc O_X(-D_2-D_7+iD_8)\\
&\to \mc O_{D_6}(-D_7+iD_8)\to 0,
\end{align*}
\[
0\to \mc O_X(-D_6-D_7+iD_8)\to \mc O_X(-D_7+iD_8)\to \mc O_{D_6}(-D_7+iD_8)\to 0.
\]
Since we can see from  (\ref{eqn:linear}) that  
$$
\mc O_X(-D_2-D_6-D_7+iD_8)\cong\mc O_X(-D_4-D_5-D_6+(i+1)D_8)\in\Span{\mk D},
$$
$$
\mc O_X(-D_2-D_7+iD_8)\cong\mc O_X(-D_5-D_6+(i+1)D_8) 
\in\Span{\mk D},
$$
$$
\mc O_X(-D_6-D_7+iD_8)\in\Span{\mk D},
$$
we know that 
$
\mc O_X(-D_7+iD_8)\in \Span{\mk D}.
$

(iv) 
We have exact sequences:
\begin{align*}
0\to \mc O_X(-D_2-D_3-D_7+iD_8)\to& \mc O_X(-D_2-D_3+iD_8)\\
&\to \mc O_{D_7}(-D_2+iD_8)\to 0,
\end{align*}
\[
0\to \mc O_X(-D_2-D_7+iD_8)\to \mc O_X(-D_2+iD_8)\to 
\mc O_{D_7}(-D_2+iD_8)\to 0.
\]
Since we have
\begin{align*}
\mc O_X(-D_2-D_3-D_7+iD_8)\cong \mc O_X(-D_5-D_6-D_7+(i+1)D_8)\in\Span{\mk D},\\
\mc O_X(-D_2-D_3+iD_8)\cong \mc O_X(-D_5-D_6+(i+1)D_8)\in \Span{\mk D},\\
\mc O_X(-D_2-D_7+iD_8)\cong \mc O_X(-D_4-D_5+(i+1)D_8)\in \Span{\mk D}
\end{align*}
by (\ref{eqn:linear}),
we obtain 
$$
\mc O_X(-D_4-D_5+D_7+(i+1)D_8)\cong \mc O_X(-D_2+iD_8)
\in \Span{\mk D}.
$$

(v) 
We have exact sequences:
\begin{align*}
0\to \mc O_X(-D_2-D_3-D_4+iD_8)\to &\mc O_X(-D_3-D_4+iD_8)\\
&\to \mc O_{D_2}(-D_3+iD_8)\to 0,
\end{align*}
\[
0\to \mc O_X(-D_2-D_3+iD_8)\to \mc O_X(-D_3+iD_8)\to \mc O_{D_2}(-D_3+iD_8)\to 0
\]%
Since we have
\begin{align*}
\mc O_X(-D_2-D_3-D_4+iD_8)\cong \mc O_X(-D_4-D_5-D_6+(i+1)D_8)\in\Span{\mk D},\\
\mc O_X(-D_3-D_4+iD_8)\cong \mc O_X(-D_6-D_7+iD_8)\in \Span{\mk D},\\
\mc O_X(D_2-D_3+iD_8)\cong \mc O_X(-D_5-D_6+iD_8)\in \Span{\mk D}
\end{align*}
by (\ref{eqn:linear}),
we obtain 
$\mc O_X(D_4-D_6-D_7+iD_8)\cong \mc O_X(-D_3+iD_8)\in \Span{\mk D}$.

Hence we know that
$\Span{\mk D}=\Span{\mk D (\omega _X^{-3})_m}$.
\end{proof}

Then by a similar argument to one given in \S \ref{sub:11},
we obtain Proposition \ref{prop:(18)}.


\section{Birational contractions and tilting objects}\label{sec:birational}


\begin{lem}\label{lem:birational0}
Let $(f,\varphi)\colon (X,\Delta _X)\to (Y,\Delta _Y)$ be a 
$T$-equivariant extremal 
birational contraction between smooth projective toric varieties.
Choose a maximal cone $\sigma$ in $\Delta_X$ such that 
$\varphi (\sigma)$ is a cone in $\Delta_Y$.
For any $\mbi u\in P_m^n$, we denote a divisor 
$D^X_{\mbi u,\mbi 0, \sigma}$ on X (resp. $D^Y_{\mbi u,\mbi 0,\varphi(\sigma)}$ on Y)
 by $D^X_{\mbi u}$ (resp. $D^Y_{\mbi u}$).
Then we have
$f_*\mc O_X(D^X_{\mbi u})=\mc O_Y(D^Y_{\mbi u})$. In particular, 
$$
\mk D_Y=\bigl\{ f_*\mc L_X \bigm| \mc L_X \in \mk D_X \bigr\}.
$$
and
$$\mc O_X(D^X_{\mbi u})=f^*\mc O_Y(D^Y_{\mbi u})\otimes \mc O_X(aE),$$
where $a\ge 0$ and $E$ is the exceptional divisor of $f$.
\end{lem}

\begin{proof}
From the commutativity $F_m\circ f=f \circ F_m$, 
we obtain 
$$
\bigoplus _{\mbi u\in P_m^n}\mc O_Y(D_{\mbi u}^Y)= F_{m*}\mc O_Y=F_{m*}f_*\mc O_X 
=f_* F_{m*}\mc O_X 
=\bigoplus _{\mbi u\in P_m^n}f_*\mc O_X(D_{\mbi u}^X).
$$
Fix some $\mbi u_1\in P_m^n$.
By the above equalities, the canonical inclusion and projection, let us define the maps
$$
\alpha_{\mbi u'} \colon \mc O_Y(D_{\mbi u_1}^Y)
\hookrightarrow 
\bigoplus _{\mbi u\in P_m^n}\mc O_Y(D_{\mbi u}^Y)=\bigoplus _{\mbi u\in P_m^n}f_*\mc O_X(D_{\mbi u}^X)
\twoheadrightarrow 
f_*\mc O_X(D_{\mbi u'}^X)
$$
and 
$$
\beta_{\mbi u'} \colon 
f_*\mc O_X(D_{\mbi u'}^X)
\hookrightarrow 
\bigoplus _{\mbi u\in P_m^n}f_*\mc O_X(D_{\mbi u}^X)=\bigoplus _{\mbi u\in P_m^n}\mc O_Y(D_{\mbi u}^Y)
\twoheadrightarrow 
\mc O_Y(D_{\mbi u_1}^Y)
$$
for each $\mbi u'\in P_m^n$.
Then we have
$$
\sum _{\mbi u'\in P_m^n}\beta _{\mbi u'}\circ \alpha_{\mbi u'}=\id ,
$$ 
and hence $\beta _{\mbi u_2}\circ \alpha_{\mbi u_2} \ne 0$ for some $\mbi u_2$. 
Since $\End_Y(\mc O_Y(D_{\mbi u_1}^Y))=\mb C$, we obtain 
$f_*\mc O_X(D_{\mbi u_2}^X)=\mc O_Y(D_{\mbi u_1}^Y)$.
Apply a similar argument for 
$$
\bigoplus _{\mbi u\in P_m^n\backslash \{\mbi u_2\}}f_*\mc O_X(D_{\mbi u}^X)
=\bigoplus _{\mbi u\in P_m^n\backslash \{\mbi u_1\}}\mc O_Y(D_{\mbi u}^Y).
$$
Then we can conclude that for all $\mbi u\in P_m^n$ 
we have $\mbi u'\in P_m^n$ such that $f_*\mc O_X(D_{\mbi u}^X)=\mc O_Y(D_{\mbi u'}^Y)$.

On the other hand, we have an inclusion 
$f_*\mc O_X(D_{\mbi u}^X)\hookrightarrow  \mc O_Y(D_{\mbi u}^Y)$, which is isomorphic in 
codimension one. Thus it is isomorphic. 
\end{proof}


\begin{lem}\label{lem:birational1}
In the situation of Lemma $\ref{lem:birational0}$, 
assume that $f$ is a equivariant blowing-up along a $T$-invariant smooth center $C$,
and define $d:=\dim E-\dim C$, $n:=\dim X$ and $\mc L_Y:=f_*\mc L_X$ 
for some $\mc L_X\in \mk D_X$.
Consider the Leray spectral sequence
\begin{align*}
 E_2^{p, q}=& H^p(Y,\mc L_Y^{\otimes m}\otimes \omega_Y \otimes \R^q f_*\mc O_X((ma+d)E))\\ 
 \Longrightarrow
 E^{p+q}=&H^{p+q}(X,f^*(\mc L_Y^{\otimes m}\otimes \omega_X)\otimes \mc O_X((ma+d)E))
\end{align*}
and assume furthermore that the vanishing
\begin{align}\label{ali:vanishingX0}
\Hom _X^i(\mc L_X,F_{m*}\mc O_X)=0
\end{align}
holds for all $i>0$. 
\begin{enumerate}
\item
The vanishing
$$\Hom_Y^i(\mc L_Y,F_{m*}\mc O_Y)=0$$
holds for all $i>0$ if and only if
$E_2^{p,d}=0$ for all $p<n-d-1=\dim C$.

In particular, if $d=n-1$, namely if $C$ is a point,
this is automatically true.
\item
Assume that 
\begin{align}\label{ali:C}
H^i(C,\mc L_Y^{\otimes -m}\otimes f_*\mc O_E(l-d-1))=0
\end{align}
for $i>0$ and all $l$ with $ma+d\ge l$, 
where we define $\mc O_E(1)$ to be the tautological line bundle of 
the $\PP ^d$-bundle $E\to C$.
Then $E_2^{p,d}=0$ for all $p$ with $p<n-d-1$.
\end{enumerate}
\end{lem}

\begin{proof}
(i) First of all, we have
\begin{align*}
&\Hom _X^i(\mc L_X,F_{m*}\mc O_X)
=\Hom _X^i(F_{m}^*\mc L_X,\mc O_X)\\
=&H^{n-i}(X,\mc L_X^{\otimes m}\otimes \omega_X )^\vee
=H^{n-i}(X,\mc L_X^{\otimes m}\otimes f^*\omega_Y\otimes \mc O_X(dE) )^\vee\\
=&H^{n-i}(X,f^*(\mc L_Y^{\otimes m}\otimes \omega_X)\otimes \mc O_X((ma+d)E))^\vee\\
=&(E^{n-i})^\vee.
\end{align*}
Hence (\ref{ali:vanishingX0}) means that 
\begin{align}\label{ali:Epq}
E^{p+q}=0
\end{align}
for all $p+q\ne n$.
Similarly it is easy to see that 
$$
(E_2^{n-i,0})^ \vee
=\Hom _Y^i(\mc L_Y,F_{m*}\mc O_Y).
$$
Therefore what we have to show is that, under the assumption (\ref{ali:Epq}),
$E_2^{p,0}=0$ for all $p< n$ is equivalent to $E_2^{p,d}=0$ for all $p<n-d-1$.
More strongly, we will see below $E_2^{p,d}\cong E_2^{p+d+1,0}$ for $p<n-d-1$.

Note that  
\begin{align}\label{ali:ElE}
\R^q f_*\mc O_E(lE)=0
\end{align}
unless $q=0,d$, since $f|_E\colon E\to C$ is a $\PP^d$-bundle.
We also have
$$
f_*\mc O_E(lE)=0
$$
for all positive $l$. Then 
we have a short exact sequence 
\begin{align}\label{ali:s.e.}
0\to \mc L_Y^{\otimes m}\otimes \omega_Y \otimes \R^q f_*\mc O_X((l-1)E)
\to &\mc L_Y^{\otimes m}\otimes \omega_Y \otimes \R^{q} f_*\mc O_X(lE)\notag\\
\to &\mc L_Y^{\otimes m}\otimes \omega_Y \otimes \R^{q} f_*\mc O_E(lE)\to 0
\end{align}
for $l\ge 0$ and all $q$. Hence by the vanishing $\R^{q} f_*\mc O_X=0$
 for $q\ne 0$ and (\ref{ali:ElE}), 
we conclude that
\begin{align*}
E_2^{p, q}=&H^p(Y,\mc L_Y^{\otimes m}\otimes \omega_Y \otimes \R^q f_*\mc O_X((ma+d)E)) \\
\cong&
H^p(Y,\mc L_Y^{\otimes m}\otimes \omega_Y \otimes \R^q f_*\mc O_X((ma+d-1)E)) \\
\cong& \cdots \\
\cong& H^p(Y,\mc L_Y^{\otimes m}\otimes \omega_Y \otimes \R^q f_*\mc O_X)=0 
\end{align*}
for all $p$ and all $q \ne 0,d$. Thus we have $E_2^{p,q}\cong E_{d+1}^{p,q}$ for all $p,q$.
Therefore from (\ref{ali:Epq}) we obtain
$$
E_2^{p,d}\cong E_{d+1}^{p,d}\cong E_{d+1}^{p+d+1,0}\cong E_2^{p+d+1,0}
$$
for $p+d+1<n$. 
Thus we obtain the conclusion.

(ii)
By the duality,
\begin{align*}
&H^{n-d-1-p}(C,\mc L_Y^{\otimes -m}\otimes f_*\mc O_E(l-d-1))^{\vee}\\
=&H^p(C,\mc L_Y^{\otimes m}\otimes (f_*\mc O_E(l-d-1))^\vee\otimes \omega_C)\\
=&H^p(Y,\mc L_Y^{\otimes m}\otimes \omega_Y \otimes \R^d f_*\mc O_E(lE)).\\
\end{align*}
By the assumption (\ref{ali:C}), the last one vanishes 
for all $l,p$ with $ma+d\ge l$ and $p<n-d-1$.
Then the vanishing of $E_2^{p,d}$
 is a direct consequence of the vanishing $\R^{d} f_*\mc O_X=0$
and (\ref{ali:s.e.}).
\end{proof}


\begin{theorem}\label{thm:surface}
Let $X$ be a toric del Pezzo surface.
Then $\mk D_X$ is a full strong exceptional collection 
on $X$.
\end{theorem}

\begin{proof}
We have already checked the statement for the maximal del Pezzo 
surface $Y_3$ in \S \ref{sub:del}. Then the statement for 
the other cases follows 
from Lemmas \ref{lem:f_*} and \ref{lem:birational1}.
\end{proof}

See also \cite[Theorem 8.2]{HP08} for an interesting result  in this direction.


\begin{lem}\label{lem:fano3}
In the notation in Lemma $\ref{lem:birational1}$,
assume that $X$ and $Y$ are smooth toric Fano $3$-folds and that the vanishing
\begin{align}\label{ali:vanishingX}
\Hom _X^i(\mc L_X,F_{m*}\mc O_X)=0
\end{align}
holds for all $i>0$. 
Then the vanishing
$$\Hom_Y^i(\mc L_Y,F_{m*}\mc O_Y)=0$$
holds for all $i>0$. 
\end{lem}

\begin{proof}
We divide the proof into two parts.
\paragraph{Step 1.}
By the last assertion in Lemma \ref{lem:birational1}(i), 
we may assume that $C\cong \PP^1$.
There are primitive generators 
$\mbi v_1,\ldots,\mbi v_5$ of $1$-dimensional cones 
in $\Delta _Y$ (and we sometimes regard them as generators of 
$1$-dimensional cones in $\Delta _X$) such that 
\begin{itemize}
\item
$C$ is the $T$-invariant curve corresponding to
the $2$-dimensional cone generated by $\mbi v_1$ and $\mbi v_4$, and
\item
the sets $\{\mbi v_1,\mbi v_4,\mbi v_5\}$,
$\{\mbi v_1,\mbi v_2,\mbi v_4\}$ and $\{\mbi v_1,\mbi v_2,\mbi v_3\}$
generate $3$-dimensional cones in $\Delta _Y$ respectively. 
\item $\mbi v_3$ is different from  $\mbi v_4$,
 but it may coincide with $\mbi v_5$. 
\end{itemize} 

We have the following equalities
$$
\mbi v_2+\mbi v_5+\alpha\mbi v_1+\beta\mbi v_4=\mbi 0 \quad
\mbox{ and }
\quad
\mbi v_3+\mbi v_4 +\gamma\mbi v_1 +\delta \mbi v_2=\mbi 0.
$$
for some $\alpha,\beta,\gamma,\delta\in \Z$
(see Figure \ref{contraction}).
Without the loss of generality, we may assume that $\beta\ge \alpha$.
Then we know that $\beta \ge 0$, since $\alpha +\beta \ge -1$ 
by the condition that $Y$ is Fano \cite[Page 89]{Od88}.
\begin{figure}[t]
\[
\unitlength 0.1in
\begin{picture}( 45.5000, 28.9500)(  1.8000,-32.8500)
%
\special{pn 8}%
\special{pa 926 574}%
\special{pa 1284 1394}%
\special{fp}%
\special{pa 1284 1394}%
\special{pa 420 1030}%
\special{fp}%
\special{pa 420 1030}%
\special{pa 920 566}%
\special{fp}%
\special{pa 920 562}%
\special{pa 1752 944}%
\special{fp}%
\special{pa 1752 944}%
\special{pa 1278 1392}%
\special{fp}%
\put(12.1700,-14.6200){\makebox(0,0){$\mbi v_1$}}%
%
\special{pn 8}%
\special{pa 926 574}%
\special{pa 1284 1394}%
\special{fp}%
\special{pa 1284 1394}%
\special{pa 420 1030}%
\special{fp}%
\special{pa 420 1030}%
\special{pa 920 566}%
\special{fp}%
\special{pa 920 562}%
\special{pa 1752 944}%
\special{fp}%
\special{pa 1752 944}%
\special{pa 1278 1392}%
\special{fp}%
\put(12.1700,-14.6200){\makebox(0,0){$\mbi v_1$}}%
%
\special{pn 8}%
\special{pa 1284 1386}%
\special{pa 2088 1436}%
\special{fp}%
\special{pa 2088 1436}%
\special{pa 1752 944}%
\special{fp}%
%
\special{pn 8}%
\special{pa 926 574}%
\special{pa 1284 1394}%
\special{fp}%
\special{pa 1284 1394}%
\special{pa 420 1030}%
\special{fp}%
\special{pa 420 1030}%
\special{pa 920 566}%
\special{fp}%
\special{pa 920 562}%
\special{pa 1752 944}%
\special{fp}%
\special{pa 1752 944}%
\special{pa 1278 1392}%
\special{fp}%
\put(12.1700,-14.6200){\makebox(0,0){$\mbi v_1$}}%
%
\special{pn 8}%
\special{pa 1284 1386}%
\special{pa 2088 1436}%
\special{fp}%
\special{pa 2088 1436}%
\special{pa 1752 944}%
\special{fp}%
\put(6.4400,-5.6000){\makebox(0,0)[lb]{$\mbi v_4$}}%
%
\special{pn 8}%
\special{pa 926 574}%
\special{pa 1284 1394}%
\special{fp}%
\special{pa 1284 1394}%
\special{pa 420 1030}%
\special{fp}%
\special{pa 420 1030}%
\special{pa 920 566}%
\special{fp}%
\special{pa 920 562}%
\special{pa 1752 944}%
\special{fp}%
\special{pa 1752 944}%
\special{pa 1278 1392}%
\special{fp}%
\put(12.1700,-14.6200){\makebox(0,0){$\mbi v_1$}}%
%
\special{pn 8}%
\special{pa 1284 1386}%
\special{pa 2088 1436}%
\special{fp}%
\special{pa 2088 1436}%
\special{pa 1752 944}%
\special{fp}%
\put(2.1000,-9.4000){\makebox(0,0)[lt]{$\mbi v_5$}}%
\put(20.1600,-8.7100){\makebox(0,0)[rt]{$\mbi v_2$}}%
\put(23.0300,-15.4500){\makebox(0,0)[rb]{$\mbi v_3$}}%
\put(6.1400,-24.1000){\makebox(0,0)[lb]{$\mbi v_4$}}%
%
\special{pn 8}%
\special{pa 896 2424}%
\special{pa 1254 3244}%
\special{fp}%
\special{pa 1254 3244}%
\special{pa 390 2880}%
\special{fp}%
\special{pa 390 2880}%
\special{pa 890 2416}%
\special{fp}%
\special{pa 890 2412}%
\special{pa 1722 2794}%
\special{fp}%
\special{pa 1722 2794}%
\special{pa 1248 3242}%
\special{fp}%
\put(11.8700,-33.1200){\makebox(0,0){$\mbi v_1$}}%
%
\special{pn 8}%
\special{pa 1254 3236}%
\special{pa 2058 3286}%
\special{fp}%
\special{pa 2058 3286}%
\special{pa 1722 2794}%
\special{fp}%
\put(1.8000,-27.9000){\makebox(0,0)[lt]{$\mbi v_5$}}%
\put(19.8600,-27.2100){\makebox(0,0)[rt]{$\mbi v_2$}}%
\put(22.7300,-33.9500){\makebox(0,0)[rb]{$\mbi v_3$}}%
\put(14.4000,-17.8000){\makebox(0,0)[lt]{$\varphi$}}%
%
\special{pn 8}%
\special{pa 1420 1514}%
\special{pa 1420 2360}%
\special{fp}%
\special{sh 1}%
\special{pa 1420 2360}%
\special{pa 1440 2292}%
\special{pa 1420 2306}%
\special{pa 1400 2292}%
\special{pa 1420 2360}%
\special{fp}%
\put(17.3000,-12.3000){\makebox(0,0)[rb]{$\sigma$}}%
\put(16.2000,-31.9000){\makebox(0,0)[lb]{$\varphi (\sigma)$}}%
%
\special{pn 8}%
\special{pa 420 1040}%
\special{pa 1070 920}%
\special{fp}%
\special{pa 1750 960}%
\special{pa 1750 960}%
\special{fp}%
%
\special{pn 8}%
\special{pa 1070 920}%
\special{pa 1080 930}%
\special{fp}%
\special{pa 1080 910}%
\special{pa 1740 950}%
\special{fp}%
%
\special{pn 8}%
\special{sh 0}%
\special{pa 1030 2930}%
\special{pa 1280 2930}%
\special{pa 1280 3106}%
\special{pa 1030 3106}%
\special{pa 1030 2930}%
\special{ip}%
\put(11.7500,-30.3500){\makebox(0,0){$\alpha$}}%
%
\special{pn 8}%
\special{sh 0}%
\special{pa 860 2530}%
\special{pa 1110 2530}%
\special{pa 1110 2706}%
\special{pa 860 2706}%
\special{pa 860 2530}%
\special{ip}%
\put(10.0500,-26.3500){\makebox(0,0){$\beta$}}%
%
\special{pn 8}%
\special{sh 0}%
\special{pa 1290 3010}%
\special{pa 1540 3010}%
\special{pa 1540 3186}%
\special{pa 1290 3186}%
\special{pa 1290 3010}%
\special{ip}%
\put(14.3500,-31.1500){\makebox(0,0){$\gamma$}}%
%
\special{pn 8}%
\special{sh 0}%
\special{pa 1470 2850}%
\special{pa 1720 2850}%
\special{pa 1720 3026}%
\special{pa 1470 3026}%
\special{pa 1470 2850}%
\special{ip}%
\put(16.1500,-29.5500){\makebox(0,0){$\delta$}}%
\put(34.2900,-6.3000){\makebox(0,0)[lb]{$\mbi v_4$}}%
\put(36.8400,-12.3700){\makebox(0,0){$\mbi v_1$}}%
\put(25.8000,-14.4000){\makebox(0,0)[lt]{$\mbi v_3=\mbi v_5$}}%
\put(48.0700,-14.8300){\makebox(0,0)[rt]{$\mbi v_2$}}%
%
\special{pn 8}%
\special{pa 3612 596}%
\special{pa 2852 1420}%
\special{fp}%
\special{pa 2852 1420}%
\special{pa 3666 1154}%
\special{fp}%
\special{pa 3666 1154}%
\special{pa 4620 1488}%
\special{fp}%
\special{pa 4620 1488}%
\special{pa 3612 592}%
\special{fp}%
\special{pa 3612 592}%
\special{pa 3662 1154}%
\special{fp}%
\special{pa 2852 1418}%
\special{pa 4616 1492}%
\special{fp}%
\put(35.3900,-24.0000){\makebox(0,0)[lb]{$\mbi v_4$}}%
\put(37.9400,-30.0700){\makebox(0,0){$\mbi v_1$}}%
\put(25.7000,-32.0000){\makebox(0,0)[lt]{$\mbi v_3=\mbi v_5$}}%
\put(49.1700,-32.5300){\makebox(0,0)[rt]{$\mbi v_2$}}%
%
\special{pn 8}%
\special{pa 3722 2366}%
\special{pa 2962 3190}%
\special{fp}%
\special{pa 2962 3190}%
\special{pa 3776 2924}%
\special{fp}%
\special{pa 3776 2924}%
\special{pa 4730 3258}%
\special{fp}%
\special{pa 4730 3258}%
\special{pa 3722 2362}%
\special{fp}%
\special{pa 3722 2362}%
\special{pa 3772 2924}%
\special{fp}%
\special{pa 2962 3188}%
\special{pa 4726 3262}%
\special{fp}%
%
\special{pn 8}%
\special{pa 3730 1570}%
\special{pa 3730 2300}%
\special{fp}%
\special{sh 1}%
\special{pa 3730 2300}%
\special{pa 3750 2234}%
\special{pa 3730 2248}%
\special{pa 3710 2234}%
\special{pa 3730 2300}%
\special{fp}%
\special{pa 3730 2300}%
\special{pa 3730 2300}%
\special{fp}%
\put(37.5000,-19.1000){\makebox(0,0)[lt]{$\varphi$}}%
%
\special{pn 8}%
\special{pa 2860 1410}%
\special{pa 2860 1410}%
\special{fp}%
\special{pa 2850 1410}%
\special{pa 3640 920}%
\special{fp}%
\special{pa 3640 920}%
\special{pa 4600 1486}%
\special{fp}%
\put(35.5000,-14.0000){\makebox(0,0)[lb]{$\sigma$}}%
\put(35.9000,-32.0000){\makebox(0,0)[lb]{$\varphi(\sigma)$}}%
%
\special{pn 8}%
\special{sh 0}%
\special{pa 3610 2470}%
\special{pa 3860 2470}%
\special{pa 3860 2646}%
\special{pa 3610 2646}%
\special{pa 3610 2470}%
\special{ip}%
\put(37.5500,-25.7500){\makebox(0,0){$\beta$}}%
%
\special{pn 8}%
\special{sh 0}%
\special{pa 3640 2700}%
\special{pa 3890 2700}%
\special{pa 3890 2876}%
\special{pa 3640 2876}%
\special{pa 3640 2700}%
\special{ip}%
\put(37.8500,-28.0500){\makebox(0,0){$\alpha$}}%
%
\special{pn 8}%
\special{sh 0}%
\special{pa 3880 2880}%
\special{pa 4130 2880}%
\special{pa 4130 3056}%
\special{pa 3880 3056}%
\special{pa 3880 2880}%
\special{ip}%
\put(40.2500,-29.8500){\makebox(0,0){$\gamma$}}%
%
\special{pn 8}%
\special{sh 0}%
\special{pa 4270 3050}%
\special{pa 4520 3050}%
\special{pa 4520 3226}%
\special{pa 4270 3226}%
\special{pa 4270 3050}%
\special{ip}%
\put(44.1500,-31.5500){\makebox(0,0){$\delta$}}%
\end{picture}%
\]
\caption{Divisorial contractions}\label{contraction} 
\end{figure}
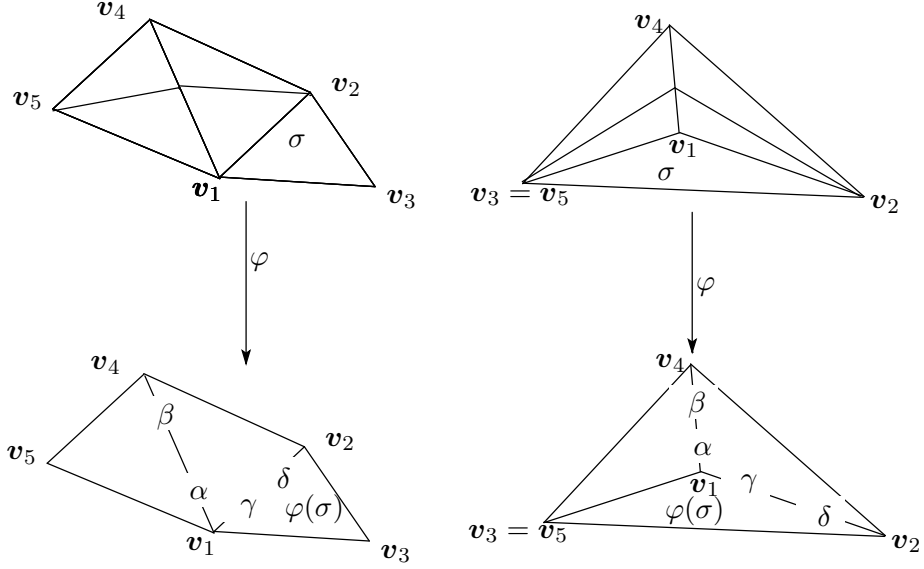
By these equalities, we have 
\begin{align}\label{ali:vs}
\mbi v_5-\beta\mbi v_3+(1-\beta\delta)\mbi v_2+(\alpha-\beta\gamma)\mbi v_1=0.
\end{align}

Note that the set $\{\mbi v_1,\mbi v_2,\mbi v_3\}$ also 
generates a $3$-dimensional 
cone, say $\sigma$, in $\Delta _X$. So we can apply 
Lemma \ref{lem:birational0} for $\sigma$.
Take $\mbi u\in P_m^3$ such that 
 $\mc L_Y\cong \mc O_X(D_{\mbi u,\varphi(\sigma)}^Y)$, and
denote by $D_i(=D_i^Y)$ (resp. $D_i^X$) the 
prime divisors on $Y$ 
(resp. $X$) corresponding to $\mbi v_i$, and 
we put $q_i$ to be the coefficient of $D_i$ in a $T$-invariant divisor 
$D_{\mbi u,\varphi(\sigma)}^Y$ on $Y$, namely we have 
$$
D_{\mbi u,\varphi(\sigma)}^Y=q_1D_1+q_2D_2+q_3D_3+q_4D_4+\cdots.
$$  
We have 
$$
D_{\mbi u,\sigma}^X= f^*D_{\mbi u,\varphi(\sigma)}^Y+aE,\quad 
\mbox{ that is } \quad 
\mc L_X\cong f^*\mc L_Y\otimes \mc O_X(aE)
$$
for some $a\ge 0$ as in Lemma \ref{lem:birational0}.
 
To check (\ref{ali:C}) in the case $C\cong \PP^1$, it is enough to show 
\begin{align}\label{ali:C1}
H^1(C,\mc L_Y^{\otimes -m}\otimes f_*\mc O_E(ma-1))=0.
\end{align}
By choosing $\mbi v_1,\mbi v_2,\mbi v_3$ as a basis of the lattice $N\cong\Z^3$, 
we obtain from (\ref{ali:vs})  that
\begin{equation}\label{equ:v_3}
\mbi v_5=\t (\beta\gamma-\alpha,\beta\delta-1,\beta)
\end{equation}
and
$q_1=q^m_1(\mbi u,\varphi(\sigma))=0, \mbox{ and } q_2=q^m_2(\mbi u,\varphi(\sigma))=0$.
We also know by (\ref{eqn:normal}) that 
$\mc N_{C/Y}\cong \mc O_C(\alpha)\oplus \mc O_C(\beta)$. 
In particular, we have
$$
f_*\mc O_E(i)\cong 
\Sym ^{i}\mc N_{C/Y}^\vee 
\cong \mc O_C(-i\alpha)\oplus \mc O_C(-(i-1)\alpha-\beta)\oplus \cdots 
\oplus \mc O_C(-i\beta) 
$$
for $i\ge 0$.
We denote $F$ a fiber of $\PP^1$-bundle 
$f|_E\colon E\to C$. Then we note that 
$T$-invariant prime divisors on $X$ which intersect with $F$ 
are only $D_1^X,D_2^X,D_4^X$ and $D_5^X$ (and of course, 
$F$ is contained in $E$).
Thus we have
$$-a=aE\cdot F= (D_{\mbi u,\sigma}^X- 
f^*D_{\mbi u,\varphi(\sigma)}^Y) \cdot F
=D_{\mbi u,\sigma}^X\cdot F =q_4,$$
since $D_2^X\cdot F=D_5^X\cdot F=0$ and $q_1=0$.
Combining this with
$$
\deg \mc L_Y|_C=(q_4D_4+q_5D_5)\cdot C
= \beta q_4+q_5,
$$
we have 
\begin{align*}
&\deg \mc L_Y^{\otimes -m}\otimes O_C((1-ma)\beta)
=-m(\beta q_4+q_5)+(1-ma)\beta \\
=&-m(\beta q_4+q_5)+(1+mq_4)\beta
=-mq_5+\beta.\\
\end{align*}
Since $\beta \ge \alpha$ and $\beta \ge 0$, we know that 
$q_5\leqq 0$ if and only if (\ref{ali:C1}) is true for $m\gg 0$.

By (\ref{equ:v_3}), we have 
$$
q_5=\lfl\frac{(\beta\gamma -\alpha )x+(\beta\delta -1)y+\beta z}{m}\rfl
$$
for $\mbi u=\t(x,y,z)\in P_m^3$.
By observing Figure \ref{pic18}, we can see that
\begin{itemize}
\item in all cases, we have $\beta\delta -1\leqq 0$, 
\item if $\beta \ge 2$,  $X$ is in 
$(11)$ and $Y$ is in $(4)$ in Theorem \ref{thm:18fano},   
\item if $\beta =1$, $\beta\gamma -\alpha\leqq 0$, and
\item if $\beta \leqq 0$ (then actually $\beta = 0$), 
 $\beta\gamma -\alpha\leqq 1$.
\end{itemize}
Consequently, we obtain $q_5\leqq 0$, except the case
 $X$ is in $(11)$
and $Y$ is in $(4)$.
\paragraph{Step 2.}
Let $X$ be the Fano $3$-fold in (11) and take $\mc L_X\in \mk D_X$. 
Then we have seen in \S \ref{sub:110} that 
$\mc L_X\not\cong \mc O_X(-D_4-D_5+D_{6})$ 
if and only if the equality
(\ref{ali:vanishingX}) 
holds for all $i>0$. 
Note that $\mbi v_6$ in Figure \ref{Fano11&18} 
plays the role of $\mbi v_5$ in Figure \ref{contraction}.
Consequently we know that if 
$\mc L_X$ is not isomorphic to $\mc O_X(-D_4-D_5+D_{6})$,
then $q_5\le 0$ by the computation in \S \ref{sub:110}. 
Then the result follows.
\end{proof}

Now we give the proof of Theorem \ref{thm:main1}.

\begin{proof}
In \S \ref{sec:Bondal} 
we have already seen that $F_{m*}\mc O_X$ is a generator 
for maximal smooth toric Fano $3$-folds $X$.
We have also seen in \S \ref{sub:17} and Proposition \ref{prop:(18)} 
that $F_{m*}\mc O_X$ is a tilting generator for the Fano $3$-folds in $(17)$ and $(18)$. Then 
Lemmas  \ref{lem:f_*} and \ref{lem:fano3} imply that $\mk D_X$ is a full 
strong exceptional collection
for all smooth toric Fano $3$-folds except the cases (4) and (11).

For the case $X$ in (11), we have seen in Proposition \ref{prop:(11)} that 
the set $\mk D_{\nef}$ 
is a full strong exceptional collection on $X$,
and in \S \ref{sub:110} that 
$$
\Hom _X^i(\mc L_X,F_{m*}\mc O_X)=0
$$
holds for all $i>0$ and all $\mc L_X\in \mk D_{\nef}$.
Take the Fano $3$-fold $Y$ in $(4)$ and consider the blowing-up $f\colon X\to Y$  
in Figure \ref{pic18a}.
Then Lemmas \ref{lem:f_*} and \ref{lem:fano3} implies that 
the subset $\bigl\{f_*\mc L_X \bigm| \mc L_X\in \mk D_{\nef}\}$
of $\mk D_Y$ 
forms a full strong exceptional collection on $Y$.
\end{proof}


\section*{Acknowledgments}
Hiroshi Sato and Yukinobu Toda are always generous with their
knowledge and ideas.
A part of the paper was written during my stay in Max-Planck Institute in September, 2010.
I appreciate all of them for their support.
I am also supported by the Grants-in-Aid 
for Scientific Research (No.23340011).


\noindent
Hokuto Uehara

Department of Mathematics
and Information Sciences,
Tokyo Metropolitan University,
1-1 Minamiohsawa,
Hachioji-shi,
Tokyo,
192-0397,
Japan 

{\em e-mail address}\ : \  hokuto@tmu.ac.jp

\end{document}